\pdfoutput=1
\RequirePackage{ifpdf}
\ifpdf 
\documentclass[pdftex]{sigma}
\else
\documentclass{sigma}
\fi

\usepackage{galois}
\usepackage{mathrsfs}
\usepackage{extarrows}
\usepackage[all]{xy}

\numberwithin{equation}{section}

\newtheorem{Theorem}{Theorem}[section]
\newtheorem{Corollary}[Theorem]{Corollary}
\newtheorem{Lemma}[Theorem]{Lemma}
\newtheorem{Proposition}[Theorem]{Proposition}
\newtheorem{Question}[Theorem]{Question}
\newtheorem{Conjecture}[Theorem]{Conjecture}

{\theoremstyle{definition}
\newtheorem{Definition}[Theorem]{Definition}

\newtheorem{Example}[Theorem]{Example}
\newtheorem{Remark}[Theorem]{Remark}}

\begin{document}
\allowdisplaybreaks

\newcommand{\arXivNumber}{1908.08348}

\renewcommand{\thefootnote}{}

\renewcommand{\PaperNumber}{029}

\FirstPageHeading

\ShortArticleName{Non-Abelian Hodge Theory and Related Topics}

\ArticleName{Non-Abelian Hodge Theory and Related Topics\footnote{This paper is a~contribution to the Special Issue on Integrability, Geometry, Moduli in honor of Motohico Mulase for his 65th birthday. The full collection is available at \href{https://www.emis.de/journals/SIGMA/Mulase.html}{https://www.emis.de/journals/SIGMA/Mulase.html}}}

\Author{Pengfei HUANG~$^{\dag\ddag}$}

\AuthorNameForHeading{P.~Huang}

\Address{$^\dag$~School of Mathematics, University of Science and Technology of China, Hefei 230026, China} \EmailD{\href{mailto:pfhwang@mail.ustc.edu.cn}{pfhwang@mail.ustc.edu.cn}}

\Address{$^\ddag$~Laboratoire J.A.~Dieudonn\'e, Universit\'e C\^ote d'Azur, CNRS, 06108 Nice, France}
\EmailD{\href{mailto:pfhwang@unice.fr}{pfhwang@unice.fr}}

\ArticleDates{Received August 23, 2019, in final form April 08, 2020; Published online April 19, 2020}

\Abstract{This paper is a survey aimed on the introduction of non-Abelian Hodge theory that gives the correspondence between flat bundles and Higgs bundles. We will also introduce some topics arising from this theory, especially some recent developments on the study of the relevant moduli spaces together with some interesting open problems.}


\Classification{14D20; 14D21; 32G20; 53C07; 57N80}

\renewcommand{\thefootnote}{\arabic{footnote}}
\setcounter{footnote}{0}

\section{Introduction}

The non-Abelian Hodge theory, which is mainly based on the work of Corlette \cite{KC} and Do\-nald\-son~\cite{SD} on harmonic maps (flat bundles), and the work of Hitchin \cite{NH} and Simpson \cite{CS1} on Higgs bundles, gives a correspondence between semisimple flat bundles over a compact K\"ahler mani\-fold~$X$ and polystable Higgs bundles over the same manifold with vanishing Chern classes. This correspondence can be generalized to a statement concerning $\lambda$-flat bundles, a topic introduced by Deligne~\cite{D} and further developed by Simpson~\cite{CS6,CS7,CS8}. More precisely, for arbitrary $\lambda_1,\lambda_2\in\mathbb{C}$, Mochizuki~\cite{TM1} established a correspondence between the categories of polystable $\lambda_1$-flat bundles with vanishing Chern classes and polystable $\lambda_2$-flat bundles with vanishing Chern classes. In particular, when $\lambda_1=1$ and $\lambda_2=0$, this correspondence recovers the original non-Abelian Hodge correspondence. Therefore, in our setting, especially in this paper, non-Abelian Hodge theory means a correspondence between flat bundles, Higgs bundles and $\lambda$-flat bundles. All these objects are connected by the existence of pluri-harmonic metrics (see Section~\ref{sec2.1}), which produce a new object: harmonic bundle.

All the work mentioned above arise from the study of a classical problem in the non-Higgs setting: a correspondence between the existence of certain special metrics on a vector bundle and the stability of that bundle.\footnote{We usually call this correspondence the {\em Kobayashi--Hitchin correspondence}.} This correspondence builds a bridge between algebraic geometric side of stability and differential geometric side of existence of pluri-harmonic metrics. The study of this kind of problem can be dated back to Narasimhan and Seshadri's work on the stability of vector bundles \cite{NS}. Their theorem states that a vector bundle over a compact Riemann surface is stable if and only if it arises from an irreducible projectively unitary representation of the fundamental group of that Riemann surface. The Narasimhan--Seshadri theorem was latter reproved by Donaldson with a differential geometric method~\cite{SD1}, which relates the stability of vector bundles and the existence of certain special metrics. This celebrating idea was later generalized to higher-dimensional compact K\"ahler manifolds by Donaldson \cite{SD2}, Uhlenbeck and Yau \cite{UY}.

There is a natural filed arising from non-Abelian Hodge theory, the study of the corresponding moduli spaces of these objects. From Simpson's work on the construction of the moduli spaces \cite{CS4,CS5,CS6}, we have the following four moduli spaces:
\begin{itemize}\itemsep=0pt
\item {\em Betti moduli space} $M_{\rm B}(X,r)$: the moduli space of rank $r$ representations $\pi_1(X)\!\to\! {\rm GL}(r,\mathbb{C})$;
\item {\em de Rham moduli space} $M_{\rm dR}(X,r)$: the moduli space of rank $r$ flat bundles over~$X$;
\item {\em Dolbeault moduli space} $M_{\rm Dol}(X,r)$: the moduli space of semistable rank~$r$ Higgs bundles over $X$ with vanishing Chern classes;
\item {\em Hodge moduli space} $M_{\rm Hod}(X,r)$: the moduli space of semistable rank~$r$ $\lambda$-flat bundles over~$X$ with vanishing Chern classes.
\end{itemize}
The study of these moduli spaces arising from non-Abelian Hodge theory shows that they are also related. More precisely, the Riemann--Hilbert correspondence implies that $M_{\rm B}(X,r)$ and $M_{\rm dR}(X,r)$ are analytic isomorphic. Moreover, the Hodge moduli space $M_{\rm Hod}(X,r)$ has a fibration over $\mathbb{C}$ such that the fibers over 0 and 1 are exactly $M_{\rm Dol}(X,r)$ and $M_{\rm dR}(X,r)$, respectively. The underlying topological spaces of $M_{\rm Dol}(X,r)$ and $M_{\rm dR}(X,r)$ are homeomorphic, and are $C^\infty$ isomorphic over the stable points.

This paper is a survey aimed on the introduction of non-Abelian Hodge theory and some related topics, especially some recent developments on the study of the moduli spaces $M_{\rm dR}(X,r)$, $M_{\rm Dol}(X,r)$ and $M_{\rm Hod}(X,r)$. Meanwhile, some open problems will be introduced, these conjectures are very interesting and important on the deep study of moduli spaces. The non-Abelian Hodge theory first came to the world in 1980s, and obtained a vast development over past decades. Many branches of mathematics are shown to be related to this theory, and this theory also sheds light on the study of other topics. During the development of this theory, there are many good references written on the introduction of this theory from different points of view, for example, \cite{DA,ABCKT,LB,BC,DP,GR,Li,SR}. There are also many important applications of this theory, for example, this theory is shown to be a powerful tool on the study of higher Teichm\"uller theory \cite{OG,Wie}. Recently, in \cite{FR} the authors applied the non-Abelian Hodge theory to study the theory of vector valued modular forms and showed a conjecture on three-termed inequality. For the application of this theory on the study of $M_{\rm B}(X,r)$ (here $X$ is a compact Riemann surface), a~recent survey paper by Migliorini~\cite{LM} introduces a kind of compactification of $M_{\rm B}(X,r)$ by dual boundary complex. The geometric $P=W$ conjecture arises from this compactification, which is a~geometric analogue of the cohomological $P=W$ conjecture on the identification between the perverse filtration on the cohomology of $M_{\rm Dol}(X,r)$, and the weight filtration on the cohomology of~$M_{\rm B}(X,r)$.

This survey is organized as follows. In the following section, we first give a brief introduction to non-Abelian Hodge theory by introducing four objects: Higgs bundles, flat bundles, $\lambda$-flat bundles, and harmonic bundles. This is based on the work of Corlette~\cite{KC}, Donaldson~\cite{SD}, Hitchin \cite{NH}, Simpson~\cite{CS1} and Mochizuki~\cite{TM1}. Then we introduce a recent development of this theory, that is the systematic work by Mochizuki on the Kobayashi--Hitchin type correspondence between periodic monopoles and difference modules~\cite{TM3,TM4,TM5}. The next three sections can be treated as various applications of non-Abelian Hodge theory, we will fix on compact Riemann surfaces. In the third section, we introduce the conformal limit conjecture proposed by Gaiotto~\cite{Gai}, which identifies two objects in the corresponding moduli spaces. More precisely, it relates the Hitchin section in~$M_{\rm Dol}$, and the space of opers in~$M_{\rm dR}$. This conjecture was recently confirmed by the authors of~\cite{DFKMMN}, and with some generalization by the authors of~\cite{CW}. The fourth section is aimed on the study of moduli spaces by generalizing the natural $\mathbb{C}^*$-action on~$M_{\rm Dol}$ to the $\mathbb{C}^*$-action on $M_{\rm Hod}$. This is based on Simpson's work~\cite{CS8}, which gives the Bia{\l}ynicki-Birula stratification of~$M_{\rm Hod}(X,r)$ into locally closed subsets. In particular, it gives the {\em oper stratification} of~$M_{\rm dR}(X,r)$, and recovers the Bia{\l}ynicki-Birula stratification of~$M_{\rm Dol}(X,r)$. The Hitchin sections and opers introduced in the third section, appears as a special stratum of the stratifications of $M_{\rm Dol}$ and $M_{\rm dR}$, respectively. Some interesting conjectures on moduli spaces proposed by Simpson \cite{CS8} are introduced in this section. Moreover, we give an explicit description between the Harder--Narasimhan filtration and the Simpson filtration for rank~3 flat bundles. The last section is aimed on the study of twistor spaces. We first introduce the Hitchin's twistor construction for general hyper-K\"ahler manifolds, which applies in particular to the moduli spaces with hyper-K\"ahelr structure. Then we introduce Deligne's interpretation of Hitchin twistor space in terms of flat $\lambda$-connections, which produces a new twistor space which is analytic isomorphic to the old one. But Deligne's construction is from a different viewpoint, that is, by gluing the moduli space~$M_{\rm Hod}(X,r)$ over~$X$ and the moduli space $M_{\rm Hod}(\bar{X},r)$ over the conjugate~$\bar{X}$. Finally, we introduce our recent work \cite{H} on a new interpretation of the Deligne--Hitchin twistor space. It can be treated as a generalization of Deligne's construction for the case of compact Riemann surfaces. We define de Rham sections for this twistor space, and show they have weight~1 property, which implies the twistor space contains ample rational curves. As an application, we obtain a Torelli-type theorem for this twistor space.

\section{Non-Abelian Hodge theory}

\subsection[Higgs bundles, flat bundles and $\lambda$-flat bundles]{Higgs bundles, flat bundles and $\boldsymbol{\lambda}$-flat bundles}\label{sec2.1}

Let $(X,\omega)$ be a compact K\"ahler manifold, where $\omega$ is the K\"ahler form. We have three different ways to define a harmonic bundle: from Higgs bundles, from flat bundles, and from $\lambda$-flat bundles. The non-Abelian Hodge theory tells us in fact they are equivalent, that is, the non-Abelian Hodge theory gives a correspondence between these objects.

To begin this theory, we first give the following notations for reader's convenience:
\begin{itemize}\itemsep=0pt
\item $C^\infty(X,\mathbb{C})$: the space of smooth complex valued functions on $X$;
\item $E$: a complex vector bundle over $X$;
\item $T^*_\mathbb{C}X$: complexified cotangent bundle, it has a decomposition
\begin{gather*}
T^*_\mathbb{C}X=(T^*X)^{1,0}\bigoplus(T^*X)^{0,1};
\end{gather*}
\item $\Omega^{p,q}_X$: the exterior algebra bundle of $(p,q)$-type, that is,
\begin{gather*}
\Omega^{p,q}_X=\bigwedge\nolimits^{\!p}\big((T^*X)^{1,0}\big) \bigotimes\bigwedge\nolimits^{\!q}\big((T^*X)^{0,1}\big);
\end{gather*}
\item $\mathcal{A}^k(X,E)$: the space of smooth complex $k$-forms on $X$ with values in $E$, that is,
\begin{gather*}
\mathcal{A}^k(X,E)=C^\infty(X,E\bigotimes\bigwedge\nolimits^{\!k}(T^*_\mathbb{C}X));
\end{gather*}
\item $\mathcal{A}^{p,q}(X,E)$: the space of smooth $(p,q)$-forms on $X$ with values on $E$, that is,
\begin{gather*}
\mathcal{A}^{p,q}(X,E)=C^\infty\big(X,E\bigotimes\Omega^{p,q}_X\big).
\end{gather*}
\end{itemize}
Then we have the decomposition
\begin{gather*}
\mathcal{A}^k(X,E)=\bigoplus_{p+q=k}\mathcal{A}^{p,q}(X,E).
\end{gather*}

\subsubsection{Higgs bundles and harmonic bundles}
A {\em holomorphic structure} on $E$ is a $\mathbb{C}$-linear operator $\bar{\partial}_E\colon \mathcal{A}^0(X,E)\to\mathcal{A}^{0,1}(X,E)$ such that
\begin{enumerate}\itemsep=0pt
\item[(1)] $\bar{\partial}_E(fs)=\bar{\partial}(f)s+f\bar{\partial}_E(s)$ for any $f\in C^\infty(X,\mathbb{C})$ and $s\in\mathcal{A}^0(X,E)$,
\item[(2)] $(\bar{\partial}_E)^2=0$;
\end{enumerate}
where the second condition is obtained under the natural extension
$\bar{\partial}_E\colon \mathcal{A}^{p,q}(X,E)\to$ \linebreak $\mathcal{A}^{p,q+1}(X,E)$
for  any integers $p, q\geq0$. Such a pair $\big(E,\bar{\partial}_E\big)$ is called a {\em holomorphic vector bundle}, sometimes we use a single $\mathcal{E}$ to denote a holomorphic vector bundle. By Koszul--Malgrange theo\-rem, this is equivalent to the usual definition of a holomorphic vector bundle by holomorphic transition functions.

\begin{Definition}A {\em Higgs bundle} over $X$ is a holomorphic vector bundle $\big(E,\bar{\partial}_E\big)$ together with a~map $\theta\colon E\to E\otimes\Omega_X^1$ which is holomorphic and integrable, i.e., $\bar{\partial}_E(\theta)=0$ and $\theta\wedge\theta=0$. Such a~map $\theta$ is called a {\em Higgs field} and the triple $(E,\bar{\partial}_E,\theta)$ denotes a Higgs bundle.
\end{Definition}

\begin{Definition}\label{def2.2}
A Higgs bundle $(E,\bar{\partial}_E,\theta)$ is called {\em stable} (resp.\ {\em semistable}) if for any proper coherent subsheaf $F$ of $0<\operatorname{rk}(F)<\operatorname{rk}(E)$ and $\theta(F)\subseteq\theta\otimes\Omega_X^1$ such that $E/F$ is torsion-free, we have
\begin{align*}
\mu(F)<  (\text{resp.} \ \leq)\ \mu(E),
\end{align*}
where $\mu(E):= \frac{\deg(E)}{\operatorname{rk}(E)}$ denotes the {\em slope} of $E$, and $\deg(E):=\int_Xc_1(E)\cdot [\omega]^{n-1}$ is the {\em degree} of~$E$. It is called {\em polystable} if it is the direct sum of stable Higgs bundles of the same slope as~$\mu(E)$.
\end{Definition}

Given a Hermitian metric $h$ on the Higgs bundle $\big(E,\bar{\partial}_E,\theta\big)$, then $h$ and $\bar{\partial}_E$ uniquely determine an $(1,0)$-type operator $\partial_{E,h}$ such that $\nabla_h:=\partial_{E,h}+\bar{\partial}_E$ is a unitary connection. Here a {\em connection}~$\nabla$ on a vector bundle~$E$ is an operator $\nabla\colon \mathcal{A}^0(X,E)\to\mathcal{A}^1(X,E)$ that satisfies the Leibniz rule $\nabla(fs)=df\otimes s+f\nabla(s)$ for any $f\in C^\infty(X,\mathbb{C})$ and $s\in\mathcal{A}^0(X,E)$. A~connection~$\nabla$ is said to be {\em unitary} with respect to the metric $h$ if it preserves this metric, i.e., if $d h(u,v)=h(\nabla(u),v)+h(u,\nabla(v))$ for any $u,v\in\mathcal{A}^0(X,E)$. This unique unitary connection is called the {\em Chern connection}, and its curvature $F_h:=\nabla_h\comp\nabla_h$ is called the {\em Chern curvature}. A~connection $\nabla$ on $E$ is said to be {\em flat} if its curvature $F_\nabla:=\nabla\comp\nabla$ vanishes under the natural extension $\nabla\colon \mathcal{A}^k(X,E)\to\mathcal{A}^{k+1}(X,E)$ for any integer $k\geq0$.

With $h$, the Higgs field $\theta$ determines an adjoint operator $\theta_h^\dagger\in\mathcal{A}^{0,1}(X,\operatorname{End}(E))$ by
\begin{gather*}
h(\theta(u),v)=h\big(u,\theta_h^\dagger(v)\big)
\end{gather*}
for any $u,v\in \mathcal{A}^0(X,E)$. The new obtained connection $\mathbb{D}^1:=\partial_{E,h}+\bar{\partial}_E +\theta+\theta_h^\dagger$ is usually called a~{\em Hitchin--Simpson connection}, and its curvature $F_{(\bar{\partial}_E,\theta,h)}:=\mathbb{D}^1\comp\mathbb{D}^1$ is called a {\em Hitchin--Simpson curvature}.

\begin{Definition}$h$ is called a {\em pluri-harmonic metric} on the Higgs bundle $\big(E,\bar{\partial}_E,\theta\big)$ if the Hitchin--Simpson connection $\mathbb{D}^1:=\partial_{E,h}+\bar{\partial}_E+\theta+\theta_h^\dagger$ is flat, that is, if $\big(E,\mathbb{D}^1\big)$ is a flat bundle. A~Higgs bundle with a pluri-harmonic metric is called a {\em harmonic bundle}, and denoted as $\big(E,\bar{\partial}_E,\theta,h\big)$.
\end{Definition}

\subsubsection{Flat bundles and harmonic bundles}

A {\em flat bundle} over $X$ is a pair $(E,\nabla)$ that consists of a complex vector bundle $E$ and a flat connection $\nabla$.

\begin{Definition}
A flat bundle $(E,\nabla)$ is called {\em irreducible} if it has no non-zero proper flat subbundle, and it is called {\em semisimple} if it is a direct sum of irreducible flat bundles.
\end{Definition}

Let $(E,\nabla)$ be a flat bundle over $X$ with a Hermitian metric $h$, then $h$ induces a unique decomposition of $\nabla$:
\begin{align*}
\nabla=\nabla_h+\Phi_h
\end{align*}
such that:
\begin{itemize}\itemsep=0pt
\item $\nabla_h$ is a unitary connection;
\item $\Phi_h\in\mathcal{A}^1(X,\operatorname{End}(E))$ is a self-adjoint operator.
\end{itemize}
To see the decomposition, we first decompose $\nabla$ into its $(1,0)$ and $(0,1)$ parts: $\nabla=\nabla^{1,0}+\nabla^{0,1}$. Then define a unique $(0,1)$-type operator $\delta_{E,h}''$ such that $\nabla^{1,0}+\delta_{E,h}''$ preserves~$h$ and a unique $(1,0)$-type operator $\delta_{E,h}'$ such that $\nabla^{0,1}+\delta_{E,h}'$ preserves~$h$. Finally, let
\begin{alignat*}{3}
&\partial_{E,h}=\frac{1}{2}\big(\nabla^{1,0}+\delta_{E,h}'\big),\qquad && \bar{\partial}_{E,h}=\frac{1}{2}\big(\nabla^{0,1}+\delta_{E,h}''\big),&\\
& \theta_{E,h}=\frac{1}{2}\big(\nabla^{1,0}-\delta_{E,h}'\big),\qquad && \theta_{V,h}^\dagger=\frac{1}{2}\big(\nabla^{0,1}-\delta_{E,h}''\big),&
\end{alignat*}
then $\nabla_h:=\partial_{E,h}+\bar{\partial}_{E,h}$ and $\Phi_h:=\theta_{E,h}+\theta_{E,h}^\dagger$ is the desired decomposition.

\begin{Definition}$h$ is called a {\em pluri-harmonic metric} on the flat bundle $(E,\nabla)$ if
\[ \big(\bar{\partial}_{E,h}+\theta_{E,h}\big)^2=0,
\] that is, if $\big(E,\bar{\partial}_{E,h},\theta_{E,h}\big)$ is a Higgs bundle. A flat bundle with a pluri-harmonic metric is called a {\em harmonic bundle}, and denoted as $(E,\nabla,h)$.
\end{Definition}

\subsubsection[$\lambda$-flat bundles and harmonic bundles]{$\boldsymbol{\lambda}$-flat bundles and harmonic bundles}

For $\lambda\in\mathbb{C}$, the notion of flat $\lambda$-connection as the interpolation of usual flat connection and Higgs field was suggested by Deligne~\cite{D}, illustrated by Simpson in~\cite{CS6} and further studied in~\cite{CS7,CS8}.

Let $\mathcal{E}=\big(E,\bar{\partial}_E\big)$ be a holomorphic vector bundle over~$X$. The following definitions of $\lambda$-flat bundles mainly come from~\cite{CS6} and~\cite{TM1} (see also~\cite{H}).

\begin{Definition}\quad
\begin{enumerate}\itemsep=0pt
 \item[(1)] A \emph{holomorphic $\lambda$-connection} on $\mathcal{E}$ is a $\mathbb{C}$-linear map $D^\lambda\colon \mathcal{E}\to \mathcal{E}\otimes\Omega_X^{1}$ that satisfies the following $\lambda$-twisted Leibniz rule:
 \begin{gather*}
 D^\lambda(fs)=fD^\lambda s+\lambda s\otimes df,
 \end{gather*}
where $f$ and $s$ are holomorphic sections of $\mathcal{O}_X$ and $E$, respectively. If $\big(D^\lambda+\bar{\partial}_E\big)\comp \big(D^\lambda+\bar{\partial}_E\big)=0$ under the natural extension $D^\lambda\colon \mathcal{E}\otimes\Omega_X^{p}\to \mathcal{E}\otimes\Omega_X^{p+1}$ for any integer $p\geq0$, then we call $D^\lambda$ a (holomorphic) \emph{flat $\lambda$-connection}, and $\big(\mathcal{E},D^\lambda\big)$ a (holomorphic) \emph{$\lambda$-flat bundle}.
 \item[(2)] A \emph{$C^\infty$ $\lambda$-connection} on $E$ is a $\mathbb{C}$-linear map $\mathbb{D}^\lambda\colon \mathcal{A}^0(X,E)\to \mathcal{A}^1(X,E)$ that satisfies the following $\lambda$-twisted Leibniz rule:
 \begin{align*}
 \mathbb{D}^\lambda(fs)=f\mathbb{D}^\lambda s+\lambda s\otimes\partial f+s\otimes\bar{\partial}f,
 \end{align*}
where $f\in C^\infty(X,\mathbb{C})$ and $s\in\mathcal{A}^0(X,E)$. If $\mathbb{D}^\lambda\comp\mathbb{D}^\lambda=0$ under the natural extension $\mathbb{D}^\lambda\colon \mathcal{A}^p(X,E)\to \mathcal{A}^{p+1}(X,E)$ for any integer $p\geq0$, then we call $\mathbb{D}^\lambda$ a ($C^\infty$) \emph{flat $\lambda$-connection}, and $\big(E,\mathbb{D}^\lambda\big)$ a ($C^\infty$) \emph{$\lambda$-flat bundle}.
\end{enumerate}
\end{Definition}

Clearly, $\lambda=1$ and $0$ correspond to the usual flat connection and Higgs field, respectively. If we work on a compact Riemann surface, then every $\lambda$-connection is automatically flat. Giving a holomorphic flat $\lambda$-connection~$D^\lambda$ on~$\mathcal{E}$ is equivalent to giving a $C^\infty$ flat $\lambda$-connection $\mathbb{D}^\lambda$ on~$E$. From now on, we will denote a $\lambda$-flat bundle as $\big(E,\bar{\partial}_E,D^\lambda\big)$ in holomorphic category and $\big(E,\mathbb{D}^\lambda\big)$ in $C^\infty$ category.

\begin{Definition}A $\lambda$-flat bundle $\big(E,\mathbb{D}^\lambda\big)$ over $X$ is called \emph{stable} (resp.\ \emph{semistable}) if for any $\lambda$-flat subbundle $\big(V,\mathbb{D}^\lambda|_V\big)$ of $0<\operatorname{rk}(V)<\operatorname{rk}(E)$, we have the following inequality
\begin{gather*}
\mu(V)< (\text{resp.} \leq) \  \mu(E),
\end{gather*}
where $\mu(E)$ is the slope of the bundle $E$ as in Definition~\ref{def2.2}. It is \emph{polystable} if it decomposes as a direct sum of stable $\lambda$-flat bundles with the same slope.
\end{Definition}

Let $\big(E,\mathbb{D}^\lambda\big)$ be a $\lambda$-flat bundle, and $h$ be a Hermitian metric on $E$. Then $h$ induces a unique decomposition of $\mathbb{D}^\lambda$:
\begin{gather*}
\mathbb{D}^\lambda=\lambda\partial_h+\theta_h+\bar{\partial}_h+\lambda\theta_h^\dagger
\end{gather*}
such that $\nabla_h :=\partial_h+\bar{\partial}_h$ is a unitary connection and $\Phi_h :=\theta_h+\theta_h^\dagger\in\mathcal{A}^1(X,\operatorname{End}(E))$ is a~self-adjoint operator.

In fact, when $\lambda=0$, this is the trivial decomposition into different types which defines a Higgs bundle structure, that is, $\mathbb{D}^0=\bar{\partial}_E+\theta$ for $\bar{\partial}_E$ defines a holomorphic structure on $E$ and $\theta$ defines a Higgs filed. When $\lambda\neq0$, we decompose $\mathbb{D}^\lambda$ into its $(1,0)$-part $d_E'$ and $(0,1)$-part~$d_E''$ that defines a holomorphic structure on~$E$.
From~$h$ and~$d_E'$, we have a (0,1)-operator~$\delta_h''$ determined by the condition
\begin{gather*}
\lambda\partial h(u,v)=h(d_E'u,v)+h(u,\delta_h''v).
\end{gather*}
Similarly, $h$ and $d_E''$ provides a $(1,0)$-operator $\delta_h'$ via the condition
\begin{gather*}
\bar{\partial}h(u,v)=h(d_E''u,v)+h(u,\delta_h'v).
\end{gather*}
One easily checks that
\begin{gather*}
\delta_h'(fs) =f\delta'_h(s)+\partial(f)\otimes s,\\
\delta_h''(fs) =f\delta_h''(s)+\bar{\lambda}\bar{\partial}(f)\otimes v
\end{gather*}
for any $f\in C^\infty(X,\mathbb{C})$ and $s\in\mathcal{A}^0(X,E)$.
We introduce the following four operators
\begin{alignat*}{3}
& \partial_h :=\frac{1}{1+|\lambda|^2}\big(\bar{\lambda}d_E'+\delta_h'\big),\qquad &&
\bar{\partial}_h:=\frac{1}{1+|\lambda|^2}\big(d_E''+\lambda\delta_h''\big),&\\
& \theta_h :=\frac{1}{1+|\lambda|^2}\big(d_E'-\lambda\delta_h'\big),\qquad &&
\theta_h^\dagger:=\frac{1}{1+|\lambda|^2}\big(\bar{\lambda}d_E''-\delta_h''\big).&
\end{alignat*}
They satisfy that
\begin{alignat*}{3}
& d_E'=\lambda\partial_h+\theta_h,\qquad && d_E''=\bar{\partial}_h+\lambda\theta_h^\dagger,&\\
& \delta_h'=\partial_h-\bar{\lambda}\theta_h,\qquad &&
\delta_h''=\bar{\lambda}\bar{\partial}_h-\theta_h^\dagger.&
\end{alignat*}
By direct calculation, $\partial_h$ and $\bar{\partial}_h$ obey the usual Leibniz rule, so $D_h=\partial_h+\bar{\partial}_h$ is a usual connection. Moreover, $D_h$ is a unitary connection with respect to $h$. By definition, $\theta_h\in C^\infty\big(X,\operatorname{End}(E)\otimes\Omega_X^{1,0}\big)$, $\theta_h^\dagger\in C^\infty\big(X,\operatorname{End}(E)\otimes\Omega_X^{0,1}\big)$, and $\theta_h^\dagger$ is the adjoint of $\theta_h$ in the sense that $h(\theta_h(u),v)=h\big(u,\theta_h^\dagger(v)\big)$.

\begin{Definition}$h$ is called a \emph{pluri-harmonic metric} on the $\lambda$-flat bundle $\big(E,\mathbb{D}^\lambda\big)$ if $(\bar{\partial}_h+\theta_h)^2=0$, that is, if $\big(E,\bar{\partial}_h,\theta_h\big)$ is a Higgs bundle. A $\lambda$-flat bundle with a pluri-harmonic metric is called a \emph{harmonic bundle}, and denoted as $\big(E,\mathbb{D}^\lambda,h\big)$.
\end{Definition}

\subsubsection{Non-Abelian Hodge correspondence}

The non-Abelian Hodge theory, arises from the existence of pluri-harmonic metric for the objects we introduced above. This is based on the work of Donaldson~\cite{SD} and Corlette~\cite{KC} on harmonic maps (flat bundles), Hitchin~\cite{NH} and Simpson~\cite{CS1} on Higgs bundles, and Mochizuki~\cite{TM1} on $\lambda$-flat bundles.

\begin{Theorem}\label{thm2.9}
Let $(X,\omega)$ be a compact K\"ahler manifold. Then
\begin{enumerate}\itemsep=0pt
\item[$(1)$] $($Donaldson~{\rm \cite{SD}}, Corlette {\rm \cite{KC})} A flat bundle $(E,\nabla)$ over $X$ admits a pluri-harmonic metric if and only if it is semisimple;
\item[$(2)$] $($Hitchin~{\rm \cite{NH}}, Simpson~{\rm \cite{CS1})} A Higgs bundle $\big(E,\bar{\partial}_E,\theta\big)$ over $X$ admits a pluri-harmonic metric if and only if it is polystable with vanishing Chern classes;
\item[$(3)$] $($Mochizuki~{\rm \cite{TM1})} A $\lambda$-flat bundle $\big(E,\mathbb{D}^\lambda\big)$ $(\lambda\neq0)$ over~$X$ admits a pluri-harmonic metric if and only if it is polystable with vanishing Chern classes.
\end{enumerate}
Moreover, in all above cases, the pluri-harmonic metric is unique up to scalar multiplicities.
\end{Theorem}

\looseness=-1 We give a brief description of above theorem here. By the work of Corlette and Donaldson, for any given flat bundle $(E,\nabla)$ of rank $r$, let $\rho\colon \pi_1(X)\to {\rm GL}(r,\mathbb{C})$ be the associated monodromy representation. If $(E,\nabla)$ is semisimple, then there is a unique map $h_\rho\colon \tilde{X}\to {\rm GL}(r,\mathbb{C})/O(r)$ which is $\rho$-equivariant and harmonic (i.e., $h_\rho$ achieves the minimum of the energy functional $E(h_\rho):=\int_X|dh_\rho|^2d\textrm{vol}$). We call such harmonic map a harmonic metric. Then by a theorem of Siu and Sampson on harmonic maps, over any compact K\"ahler manifold, each harmonic map is pluri-harmonic, i.e., $D^{0,1}\big((df_\rho)^{0,1}\big)=0$, where~$D$ is the descend of the pull back of the canonical Levi-Civita connection on the tangent bundle of the symmetric space $ {\rm GL}(r,\mathbb{C})/O(r)$ by~$h_\rho$. Therefore, it induces a Higgs bundle structure. On the other hand, by the work of Hitchin and Simpson, for any given Higgs bundle $\big(E,\bar{\partial}_E,\theta\big)$ of rank~$r$. If it is polystable, then it admits a Hermitian--Einstein metric $h$, i.e., the Hitchin--Simpson curvature $F_{(\bar{\partial}_E,\theta,h)}$ satisfies the Hermitian--Yang--Mills equation $\Lambda_\omega F_{(\bar{\partial}_E,\theta,h)}=\alpha{\rm Id}$ for some constant~$\alpha$. Moreover, if such Higgs bundle has vanishing first and second Chern characters, then a Bogomolov--Gieseker inequality will imply $F_{(\bar{\partial}_E,\theta,h)}=0$. Therefore, it induces a flat bundle structure. This ideas were generalized by Mochizuki to study the case of $\lambda$-flat bundles \cite{TM1}, where he considered this object over higher-dimensional smooth complex projective varieties with a simple normal crossing divisor. In~\cite{TM1}, he built the Kobayashi--Hitchin correspondence between polystable parabolic $\lambda$-flat bundles with vanishing first and second parabolic Chern characters and tame harmonic bundles. Mochizuki's method relies on a $\varepsilon$-perturbation technique developed by himself on the study of higher-dimensional Kobayashi--Hitchin correspondence between polystable parabolic Higgs bundles with vanishing first and second parabolic Chern characters and tame harmonic bundles~\cite{TM}.

Therefore, we obtain an equivalence between the categories of these objects:

\begin{Corollary}[non-Abelian Hodge correspondence]\label{thm2.6}
Let $(X,\omega)$ be a compact K\"ahler manifold. Then for each $\lambda\in\mathbb{C}$, we have one to one correspondence between the equivalence classes of polystable $\lambda$-flat bundles with vanishing Chern classes, the equivalence classes of polystable Higgs bundles with vanishing Chern classes, and the equivalence classes of semisimple flat bundles, through harmonic bundles.
\end{Corollary}

\begin{Remark}\quad
\begin{enumerate}\itemsep=0pt
\item[(1)] The correspondence between harmonic $\lambda$-flat bundles and harmonic Higgs bundles is given as follows:
\begin{gather*}\begin{split}&
\big(E,\mathbb{D}^\lambda,h\big)\longmapsto\big(E,\bar{\partial}_h,\theta_h,h\big),\\
&\big(E,\bar{\partial}_E,\theta,h\big)\longmapsto\big(E,\bar{\partial}_E+\lambda\theta_h^\dagger, \lambda\partial_{E,h}+\theta,h\big),
\end{split}
\end{gather*}
and the correspondence between harmonic $\lambda$-bundles and harmonic flat bundles is given as follows:
\begin{gather*}
\big(E,\mathbb{D}^\lambda,h\big) \longmapsto\big(E,\partial_h+\bar{\partial}_h+\theta_h+\theta_h^\dagger,h\big),\\
(E,\nabla,h) \longmapsto\big(E,\bar{\partial}_{E,h}+\lambda\theta_{E,h}^\dagger,\lambda\partial_{E,h}+\theta_{E,h},h\big).
\end{gather*}

\item[(2)] A natural direction of the study of non-Abelian Hodge correspondence is the generalization to non-compact case. In regular case (tame case), Simpson established the correspondence between tame harmonic bundles and parabolic Higgs bundles for open curves~\cite{CS}. The higher-dimensional generalization is obtained by Biquard \cite{OB} for smooth divisor case and by Mochizuki \cite{TM,TM1} for general case. In irregular case (wild case), Biquard and Boalch built the correspondence for curves \cite{BB}. Later Mochizuki generalized it to higher-dimensional case \cite{TM6}. Recently, the authors in \cite{GKPT} built the correspondence to the context of projective varieties with Kawamata log terminal (brief as klt) singularities.

\item[(3)] Another natural generalization is considering the corresponding case for real Lie groups. In \cite{BGM,GGM}, the authors considered the principal $G$-Higgs bundles for real Lie group $G$. They studied the Kobayashi--Hitchin correspondence for principal $G$-Higgs bundles, and therefore, they built the non-Abelian Hodge correspondence for such objects.

\item[(4)] The non-Abelian Hodge theory for varieties over a field of characteristic $p$ was built by Ogus and Vologodsky in~\cite{OV}. The $p$-adic version was suggested by Faltings in~\cite{GF} and finished by Abbes, Gros and Tsuji in~\cite{AGT}.
\end{enumerate}
\end{Remark}

Denote by $M_{\rm dR}(X,r)$ the coarse moduli space of completely reducible flat bundles of rank $r$, and by $M_{\rm Dol}(X,r)$ the coarse moduli space of semistable Higgs bundles of rank $r$ with vanishing Chern classes. Then by the non-Abelian Hodge correspondence, we can obtain a map between moduli spaces:
\begin{align*}
{\rm NAH}\colon \ M_{\rm Dol}(X,r) &\longrightarrow M_{\rm dR}(X,r),\\
\big[E,\bar{\partial}_E,\theta\big] &\longmapsto \big[E,\bar{\partial}_E+\theta^{\dagger_h},\partial_{E,h}+\theta\big],
\end{align*}
where $h$ is the unique pluri-harmonic metric of $\big[E,\bar{\partial}_E,\theta\big]$. We call such map the {\rm non-Abelian Hodge map}, it will be used in the next two sections.

\subsection{Periodic monopoles and difference modules}

The Kobayashi--Hitchin correspondence, which relates the stability of bundles (i.e., Higgs bundles, flat bundles, $\lambda$-flat bundles) and the existence of certain special metrics (i.e., Hermitian--Einstein metrics, pluri-harmonic metrics), plays an important role on non-Abelian Hodge theory. Theorem \ref{thm2.9} is obtained under various work of Kobayashi--Hitchin correspondence.

Recently, Mochizuki built the Kobayashi--Hitchin type correspondence between periodic mo\-no\-po\-les and difference modules~\cite{TM3,TM4,TM5}, as the non-Abelian Hodge theory for monopoles with periodicity. Here we give a brief introduction following these papers.

Let $\Gamma\subset\mathbb{R}^3$ be a lattice, with the quotient space $\mathcal{M}:=\mathbb{R}^3/\Gamma$, which is equipped with a~metric~$g_\mathcal{M}$ induced from the Euclidean metric of $\mathbb{R}^3$. Let $Z\subset\mathcal{M}$ be a finite subset, let~$(E,h)$ be a Hermitian vector bundle over $\mathcal{M}\backslash Z$ with a unitary connection $\nabla$ and an anti-self-adjoint endomorphism $\phi$, which is called a Higgs filed.

\begin{Definition}The quadruple $(E,h,\nabla,\phi)$ is called a {\em monopole} on $\mathcal{M}\backslash Z$ if it satisfies the Bogomolny equation:
\begin{align*}
F(\nabla)-*\nabla\phi=0,
\end{align*}
where $F(\nabla)=\nabla\comp\nabla$ is the curvature of $\nabla$ and $*$ is the Hodge star operator with respect to $g_\mathcal{M}$. Moreover, it is called a {\em periodic monopole} (resp. {\em doubly periodic monopole}, resp. {\em triply periodic monopole}) if $\Gamma\cong\mathbb{Z}$ (resp. $\Gamma\cong\mathbb{Z}^2$, resp. $\Gamma\cong\mathbb{Z}^3$).
\end{Definition}

\subsubsection{Periodic monopoles of GCK type and parabolic difference modules}

In \cite{TM3}, Mochizuki studied and built the Kobayashi--Hitchin type correspondence between irreducible periodic monopoles with GCK type singularity and stable parabolic difference modules of degree 0. In this case, $\mathcal{M}\backslash Z=\big(T^1\times\mathbb{R}^2\big)\backslash Z=\big(T^1\times\mathbb{C}\big)\backslash Z$.

\begin{Definition}A periodic monopole $(E,h,\nabla,\phi)$ on $\mathcal{M}\backslash Z$ is called of {\em GCK $($generalized Cherkis--Kapustin$)$ type} if it satisfies the following two conditions:
\begin{enumerate}\itemsep=0pt
\item[(1)] each point of $P\in Z$ is of Dirac type singularity of the monopole, this is equivalent to say, $|{\phi_|}_Q|_h=O\big(d(P,Q)^{-1}\big)$ for any $Q$ close to $P$~\cite{MY};
\item[(2)] $|\phi|_h=O(\log|z|)$ and $|F(\nabla)|\to0$ as $|z|\to\infty$, where $z$ is the coordinate of $\mathbb{C}$.
\end{enumerate}
\end{Definition}

Let $\lambda\in\mathbb{C}$, called a {\em twistor parameter}, let $\Phi^*\in\operatorname{Aut}(\mathbb{C}[y])$ determined by $\Phi^*(f(y))=f(y+\lambda)$.

\begin{Definition}
A {\em difference module on $\mathbb{C}$} (or called {\em $\lambda$-difference module}) is a $\mathbb{C}[y]$-module $\mathbf{V}$ with a complex linear automorphism $\Phi^*$ such that $\Phi^*(fs)=\Phi^*(f)\cdot\Phi^*(s)$ for any $f\in\mathbb{C}[y]$ and $s\in \mathbf{V}$. It is {\em torsion-free} if it is torsion-free as a $\mathbb{C}[y]$-module. It is {\em of finite type} if it is finitely generated over the algebra $\mathscr{A}_\lambda=\bigoplus_{n\in\mathbb{Z}}\mathbb{C}[y](\Phi^*)^n$ and $\dim_{\mathbb{C}[y]}\mathbf{V}\otimes\mathbb{C}(y)<\infty$, where $\mathscr{A}_\lambda$ has the product determined by $\Phi^*\cdot f(y)=f(y+\lambda)\cdot\Phi^*$.
\end{Definition}

Now let $\mathbf{V}$ be a torsion-free $\lambda$-difference module of finite type.

\begin{Definition}
A {\em parabolic structure at finite place} of $\mathbf{V}$ consists of the followings:
\begin{enumerate}\itemsep=0pt
\item[(1)] A free sub $\mathbb{C}[y]$-module $V\subset\mathbf{V}$ such that $\mathscr{A}_\lambda\cdot V=\textbf{V}$ and $V\otimes_{\mathbb{C}[y]}\mathbb{C}(y)=\mathbf{V}\otimes_{\mathbb{C}[y]}\mathbb{C}(y)$;
\item[(2)] A function $m\colon \mathbb{C}\to\mathbb{Z}_{\geq0}$ with $\sum\limits_{x\in\mathbb{C}}m(x)<\infty$, we assume $V\otimes_{\mathbb{C}[y]}\mathbb{C}[y]_D=(\Phi^*)^{-1}(V)\otimes_{\mathbb{C}[y]}\mathbb{C}[y]_D$ for $D:=\{x\in\mathbb{C}\,|\, m(x)>0\}$, where $\mathbb{C}[y]_D$ denotes the localization of $\mathbb{C}[y]$ with respect to $y-x$ ($x\in D$);
\item[(3)] For each $x\in\mathbb{C}$, there is a sequence of real numbers $\mathbb{\alpha}_x:=\{0\leq\alpha_{x,1}<\cdots<\alpha_{x,m(x)}<1\}$, when $m(x)=0$, it is assumed to be empty;
\item[(4)] For each $x\in\mathbb{C}$, there is a sequence of lattices $\mathbf{L}_x:=\big\{L_{x,0}=V\otimes_{\mathbb{C}[y]}\mathbb{C}[\![y-x]\!],L_{x,i}\subset V\otimes_{\mathbb{C}[y]}\mathbb{C}(\!(y-x)\!) (1\leq i\leq m(x)-1), L_{x,m(x)}=(\Phi^*)^{-1}(V)\otimes_{\mathbb{C}[y]}\mathbb{C}[\![y-x]\!]\big\}$.
\end{enumerate}
We use the quadruple $(V,m,\{\alpha_x,\mathbf{L}_x\}_{x\in\mathbb{C}})$ to denote it.
\end{Definition}

Let $V$ be a torsion-free $\lambda$-difference module of finite type and let $\widehat{\mathbf{V}}:=\mathbf{V}\otimes_{\mathbb{C}[y]}\mathbb{C}\big(\!\big(y^{-1}\big)\!\big)$ be its formal completion at $\infty$.

\begin{Definition}
A {\em parabolic structure at $\infty$} of $\mathbf{V}$ is a filtered bundle $\mathscr{F}_*\widehat{\mathbf{V}}=\big(\mathscr{F}_a\widehat{\mathbf{V}}\,|\, a\in\mathbb{R}\big)$ over $\widehat{\mathbf{V}}$ such that:
\begin{enumerate}\itemsep=0pt
\item[(1)] For each $a\in\mathbb{R}$, $\mathscr{F}_a\widehat{\mathbf{V}}$ is a sub $\mathbb{C}\big[\!\big[y^{-1}\big]\!\big]$-module of $\widehat{\mathbf{V}}$ with $\mathscr{F}_a\widehat{\mathbf{V}}\otimes_{\mathbb{C}[y]}\mathbb{C}\big(\!\big(y^{-1}\big)\!\big) =\widehat{\mathbf{V}}$;
\item[(2)] $\mathscr{F}_a\widehat{\mathbf{V}}\subset\mathscr{F}_b\widehat{\mathbf{V}}$ if $a\leq b$;
\item[(3)] $\mathscr{F}_{a+n}\widehat{\mathbf{V}}=y^n\mathscr{F}_a\widehat{\mathbf{V}}$ for any $a\in\mathbb{R}$ and $n\in\mathbb{Z}$;
\item[(4)] For any $a\in\mathbb{R}$, there exists $\varepsilon>0$ such that $\mathscr{F}_a\widehat{\mathbf{V}}=\mathscr{F}_{a+\varepsilon}\widehat{\mathbf{V}}$.
\end{enumerate}

$\mathscr{F}_*\widehat{\mathbf{V}}$ is called {\em good} if we have a decomposition
\begin{gather*}
\mathscr{F}_*\big(\widehat{\mathbf{V}}\otimes\mathbb{C}\big(\!\big(y^{-1/p}\big)\!\big)\big)=\bigoplus_{c\in p^{-1}\mathbb{Z}}\bigoplus_{d\in\mathbb{C}^*}\bigoplus_{e\in S(p)}\mathscr{F}_*\widehat{\mathbf{V}}_{c,d,e}
\end{gather*}
such that $\big(c^{-1}y^{-c}\Phi^*-(1+e)\text{Id}\big)\mathscr{F}_a\widehat{\mathbf{V}}_{c,d,e}\subset y^{-1}\mathscr{F}_a\widehat{\mathbf{V}}_{c,d,e}$ for any $a\in\mathbb{R}$, here $p\in\mathbb{Z}_{>0}$ and $S(p):=\big\{\sum\limits_{i=1}^{p-1}e_iy^{-i/p}\,|\, e_i\in\mathbb{C}\big\}$. A {\em parabolic $\lambda$-difference module} is a torsion-free $\lambda$-difference module of finite type with a parabolic structure at finite place and a good parabolic structure at $\infty$. We denote it as the following:
\begin{gather*}
\big(\mathbf{V},(V,m,\{\alpha_x,\mathbf{L}_x\}_{x\in\mathbb{C}}),\mathscr{F}_*\widehat{\mathbf{V}}\big).
\end{gather*}
For simplicity, we will just use $(\mathbf{V},\bullet)$ to denote such an object.
\end{Definition}

\begin{Definition}
The {\em degree} of a parabolic $\lambda$-difference module $(\mathbf{V},\bullet)$ is given by
\begin{gather*}
\deg(\mathbf{V},\bullet):= \deg(\mathscr{F}_0\mathcal{V})-\sum_{-1<a\leq0}a\dim_\mathbb{C}\mathrm{Gr}_a^\mathscr{F} \big(\widehat{\mathbf{V}}\big)-\sum_{c\in\mathbb{Q}}\frac{c}{2}r(c)\\
\hphantom{\deg(\mathbf{V},\bullet):= }{} +\sum_{x\in\mathbb{C}}\sum_i(1-\alpha_{x,i})\deg(L_{x,i},L_{x,i-1}),
\end{gather*}
where $\mathscr{F}_0\mathcal{V}$ is the $\mathcal{O}_{\mathbb{P}^1}$-modules associated to the free $\mathbb{C}[y]$-module $V$, $\mathrm{Gr}_a^\mathscr{F}\big(\widehat{\mathbf{V}}\big):= \mathscr{F}_a\widehat{\mathbf{V}}/\mathscr{F}_{<a}\widehat{\mathbf{V}}$, $r(c):=\sum\limits_{d\in\mathbb{C}^*}\sum\limits_{e\in S(p)}\dim_{\mathbb{C}(\!(y^{-1/p})\!)}\widehat{\mathbf{V}}_{c,d,e}$ and if $m(x)>0$, then for $1\leq i\leq m(x)$,
\begin{gather*} \deg(L_{x,i},L_{x,i}):=\dim_\mathbb{C}(L_{x,i}/(L_{x,i}\cap L_{x,i-1}))-\dim_\mathbb{C}(L_{x,i-1}/(L_{x,i}\cap L_{x,i-1})).\end{gather*}
The {\em slope} is given by
\begin{gather*}
\mu(\mathbf{V},\bullet):=\deg(\mathbf{V},\bullet)/\operatorname{rk}_{\mathbb{C}[y]}(V).
\end{gather*}
\end{Definition}

For any $\mathbb{C}(y)$-subspace $\widetilde{\mathbf{V}}'\subset\widetilde{\mathbf{V}}:=\mathbf{V}\otimes\mathbb{C}(y)$ that is $\Phi^*$-invariant, i.e., $\phi^*\big(\widetilde{\mathbf{V}}'\big)\subset\widetilde{\mathbf{V}}'$, set $V'=\widetilde{\mathbf{V}}'\cap V$ and $\mathbf{V}'=\mathscr{A}\cdot V'$, then there is an induced structure on~$\mathbf{V}'$:
\begin{gather*}
L_{x,i}':=L_{x,i}\cap V',\qquad \widehat{\mathbf{V}}':=\mathbf{V}'\otimes_{\mathbb{C}[y]}\mathbb{C}\big(\!\big(y^{-1}\big)\!\big),\qquad \mathscr{F}_a\widehat{\mathbf{V}}':=\mathscr{F}_a\widehat{\mathbb{V}}\cap\widehat{\mathbf{V}}'.
\end{gather*}
This structure is a parabolic structure at finite place and a good parabolic structure at $\infty$, we denote the sub parabolic $\lambda$-difference module as $\big(\mathbf{V}',m,\{\alpha_x, \mathbf{L}_x'\}_{x\in\mathbb{C}},\mathscr{F}_*\widehat{\mathbf{V}}'\big)$, for simplicity, we denote it as $(\mathbf{V}',\bullet)$.

\begin{Definition}
A parabolic $\lambda$-difference module $\big(\mathbf{V},(V,m,\{\alpha_x,\mathbf{L}_x\}_{x\in\mathbb{C}}),\mathscr{F}_*\widehat{\mathbf{V}}\big)$ is called {\em stable} (resp.\ {\em semistable}) if for any non-zero proper $\mathbb{C}(y)$-subspace $\widetilde{\mathbf{V}}'$ that is $\Phi^*$-invariant, we have
$\mu(\mathbf{V}',\bullet)< (\text{resp.\ $\leq$}) \ \mu(\mathbf{V},\bullet)$.
\end{Definition}

In \cite{TM3}, Mochizuki built the following Kobayashi--Hitchin type correspondence between periodic monopoles and $\lambda$-difference modules, as the analogue of the correspondence between harmonic bundles and (poly-)stable $\lambda$-flat bundles:

\begin{Theorem}[\cite{TM3}]For each $\lambda\in\mathbb{C}$, there is an one to one correspondence between irreducible periodic monopoles of $GCK$ type and stable parabolic $\lambda$-difference modules of degree~$0$.
\end{Theorem}

\subsubsection[Meromorphic doubly periodic monopoles and parabolic $q$-difference modules]{Meromorphic doubly periodic monopoles and parabolic $\boldsymbol{q}$-difference modules}

As we have defined, when $\Gamma\cong\mathbb{Z}^2$, $\mathcal{M}=\mathbb{R}^3/\Gamma=T^2\times\mathbb{R}$ and $Z\subset\mathcal{M}$ a finite subset, a monopole $(E,h,\nabla,\phi)$ on $\mathcal{M}\backslash Z$ is called a doubly periodic monopole. In~\cite{TM4}, Mochizuki studied and built the Kobayashi--Hitchin type correspondence between irreducible meromorphic doubly periodic monopoles and stable $q$-difference modules of degree~0.

\begin{Definition}A doubly periodic monopole $(E,h,\nabla,\phi)$ on $\mathcal{M}\backslash Z$ is called {\em meromorphic} if each point of $Z$ is of Dirac type singularity and if the curvature~$F(\nabla)$ is bounded with respect to~$h$ and $g_\mathcal{M}$ when $|t|\to\infty$, where~$t$ is the coordinate for~$\mathbb{R}$.
\end{Definition}

Let $q\in\mathbb{C}^*$, called a {\em twistor parameter}, let $\Phi\in\operatorname{Aut} \big(\mathbb{C}\big[y,y^{-1}\big]\big)$ determined by $\Phi(f(y))=f(qy)$.

\begin{Definition}
A {\em difference module on $\mathbb{C}^*$} (or called {\em $q$-difference module}) is a $\mathbb{C}\big[y,y^{-1}\big]$-module~$\mathbf{V}$ with a complex linear automorphism~$\Phi^*$ such that $\Phi^*(fs)=\Phi^*(f)\cdot\Phi^*(s)$ for any $f\in\mathbb{C}\big[y,y^{-1}\big]$ and $s\in \mathbf{V}$. It is {\em torsion-free} if it is torsion-free as a $\mathbb{C}\big[y,y^{-1}\big]$-module.
\end{Definition}

Let $\mathscr{A}_q:=\bigoplus_{n\in\mathbb{Z}}\mathbb{C}\big[y,y^{-1}\big](\Phi^*)^n$ be an algebra determined by $\Phi^*\cdot f(y)=f(qy)\cdot\Phi^*$, and let $\mathbf{V}$ be a torsion-free $q$-difference module.

\begin{Definition}A {\em parabolic structure at finite place} of~$\mathbf{V}$ consists of the followings:
\begin{enumerate}\itemsep=0pt
\item[(1)] A free sub $\mathbb{C}\big[y,y^{-1}\big]$-module $V\subset\mathbf{V}$ such that $\mathscr{A}_q\cdot V=\textbf{V}$ and $V\otimes_{\mathbb{C}[y,y^{-1}]}\mathbb{C}(y)=\mathbf{V}\otimes_{\mathbb{C}[y,y^{-1}]}\mathbb{C}(y)$;
\item[(2)] A finite subset $D\subset\mathbb{C}^*$ such that $V(*D)=(\Phi^*)^{-1}(V)(*D)$, where $V(*D):=V\otimes_{\mathbb{C}[y,y^{-1}]}\mathbb{C}[y,y^{-1}](*D)$;
\item[(3)] For each $x\in D$, there is a sequence of real numbers $\mathbb{\alpha}_x:=\{0\leq\alpha_{x,1}<\cdots<\alpha_{x,m(x)}<1\}$, when $m(x)=0$, it is assumed to be empty;
\item[(4)] For each $x\in D$, there is a sequence of lattices $\mathbf{L}_x:=\big\{L_{x,0}=V\otimes_{\mathbb{C}[y,y^{-1}]}\mathbb{C}[\![y-x]\!],L_{x,i}\subset V\otimes_{\mathbb{C}[y,y^{-1}]}\mathbb{C}(\!(y-x)\!) (1\leq i\leq m(x)-1), L_{x,m(x)}=(\Phi^*)^{-1}(V)\otimes_{\mathbb{C}[y,y^{-1}]}\mathbb{C}[\![y-x]\!]\big\}$.
\end{enumerate}
We denote it as the quadruple $(V,D,\{\alpha_x,\mathbf{L}_x\}_{x\in D})$.
\end{Definition}

\begin{Definition}
A {\em good parabolic structure at $\infty$} of $\mathbf{V}$ is the pair $\big(\mathscr{F}_*{\mathbf{V}_|}_{\widehat{0}},\mathscr{F}_*{\mathbf{V}_|}_{\widehat{\infty}}\big)$, where $\mathscr{F}_*{\mathbf{V}_|}_{\widehat{0}}$ is a good filtered bundle over ${\mathbf{V}_|}_{\widehat{0}}:=\mathbf{V}\otimes\mathbb{C}(\!(y)\!)$ and $\mathscr{F}_*{\mathbf{V}_|}_{\widehat{\infty}}$ is a good filtered bundle over ${\mathbf{V}_|}_{\widehat{\infty}}:=\mathbf{V}\otimes\mathbb{C}\big(\!\big(y^{-1}\big)\!\big)$. A {\em parabolic $q$-difference module} is a $q$-difference module $\mathbf{V}$ with a parabolic structure at finite place and a good parabolic structure at $\infty$, we denote it as $\big(\mathbf{V},(V,D\{\alpha_x,\mathbf{L}_x\}_{x\in D}),\mathscr{F}_*{\mathbf{V}_|}_{\widehat{0}},\mathscr{F}_*{\mathbf{V}_|}_{\widehat{\infty}}\big)$, for simplicity, we denote it as~$(\mathbf{V},\bullet)$.

The {\em degree} and {\em slope} of a parabolic $q$-difference module $(\mathbf{V},\bullet)$ are given by
\begin{gather*}
\deg(\mathbf{V},\bullet):= \deg(\mathscr{F}_0\mathcal{V})+\sum_{x\in D}\sum_{i=1}^{m(x)}(1-\alpha_{x,i})\deg(L_{x,i},L_{x,i-1})\\
\hphantom{\deg(\mathbf{V},\bullet):=}{} -\sum_{c\in\mathbb{Q}}\frac{c}{2}\big(\big(\dim_{\mathbb{C}(\!(y)\!)}({\mathbb{V}_|}_{\widehat{0}})_c\big) +\dim_{\mathbb{C}(\!(y^{-1})\!)}(({\mathbf{V}_|}_{\widehat{\infty}})_c)\big),
\end{gather*}
and $\mu(\mathbf{V},\bullet):=\frac{\deg(\mathbf{V},\bullet)}{\operatorname{rk}_{\mathbb{C}[y,y^{-1}]}(V)}$, respectively,
where each $\big({\mathbb{V}_|}_{\widehat{0}}\big)_c$ is a term associated to the slope decomposition ${\mathbf{V}_|}_{\widehat{0}}=\bigoplus\limits_{c\in\mathbb{Q}}\big({\mathbb{V}_|}_{\widehat{0}}\big)_c$ that is $\Phi^*$-invariant and for each $c=l/k\in\mathbb{Q}$, there exists $\mathbb{C}[\![y]\!]$-lattice $L_c\subset({\mathbb{V}_|}_{\widehat{0}})_c$ with $y^l(\Phi^*)^kL_c=L_c$. It is called {\em stable} (resp. {\em semistable}) if for any non-zero proper $q$-difference $\mathbb{C}(y)$-subspace $\widetilde{\mathbf{V}}'$ of $\widetilde{\mathbf{V}}:=\mathbf{V}\otimes\mathbb{C}(y)$ with the naturally induced parabolic structure, we have $\mu(\mathbf{V}',\bullet)< (\text{resp. $\leq$}) \ \mu(\mathbf{V},\bullet)$.
\end{Definition}

In \cite{TM4}, Mochizuki built the following Kobayashi--Hitchin type correspondence between doubly periodic monopoles and $q$-difference modules:

\begin{Theorem}[\cite{TM4}]
For each $q\in\mathbb{C}^*$, there is an one to one correspondence between irreducible meromorphic doubly periodic monopoles and stable parabolic $q$-difference modules of degree~$0$.
\end{Theorem}

\subsubsection[Meromorphic triply periodic monopoles and parbolic difference modules on elliptic curves]{Meromorphic triply periodic monopoles and parbolic difference modules\\ on elliptic curves}

When $\Gamma\cong\mathbb{Z}^3$, $\mathcal{M}=\mathbb{R}^3/\Gamma=T^3$ and $Z\subset \mathcal{M}$ a finite subset, a monopole $(E,h,\nabla,\phi)$ on $\mathcal{M}\backslash Z$ is called a triply periodic monopole. In \cite{TM5}, Mochizuki studied and built the Kobayashi--Hitchin type correspondence between meromorphic triply periodic monopoles and stable difference modules on elliptic curves of degree 0.

\begin{Definition}A triply periodic monopole $(E,h,\nabla,\phi)$ on $\mathcal{M}\backslash Z$ is called {\em meromorphic} if each point of~$Z$ is of Dirac singularity.
\end{Definition}

Let $\Gamma_0\subset\mathbb{C}$ be a lattice and let $T:=\mathbb{C}/\Gamma_0$ be the elliptic curve, let $a\in T$, called a twistor parameter, let $\Phi\in\operatorname{Aut}(T)$ given by $\Phi(z)=z+a$, let $D\subset T$ be a finite subset.

\begin{Definition}A {\em parabolic difference module on the elliptic curve $T$} (or {\em parabolic $a$-difference module}) consists of the followings:
\begin{enumerate}\itemsep=0pt
\item[(1)] A locally free $\mathcal{O}_T$-module $V$ and an isomorphism of $\mathcal{O}_T$-modules $V(*D)\cong(\Phi^*)^{-1}(V)(*D)$;
\item[(2)] For each $x\in D$, there is a sequence of real numbers $\mathbb{\alpha}_x:=\{0\leq\alpha_{x,1}<\cdots<\alpha_{x,m(x)}<1\}$, when $m(x)=0$, it is assumed to be empty;
\item[(3)] For each $x\in D$, there is a sequence of lattices $\mathbf{L}_x:=\big\{L_{x,0}=V_x,L_{x,i}\subset V(*D)_x (1\leq i\leq m(x)-1), L_{x,m(x)}=(\Phi^*)^{-1}(V)_x\big\}$.
\end{enumerate}
We denote it as the quadruple $\big(V,D,\{\alpha_x,\mathbf{L}_x\}_{x\in D}\big)$, and for simplicity, we denote it as $(V,\bullet)$. The {\em degree} and {\em slope} of $(V,\bullet)$ are given by
\begin{align*}
\deg(V,\bullet):=\deg(V)+\sum_{x\in D}\sum_{i=1}^{m(x)}(1-\alpha_{x,i})\deg(L_{x,i},L_{x,i-1})
\end{align*}
and $\mu(V,\bullet):=\frac{\deg(V,\bullet)}{\operatorname{rk}(V)}$, respectively. It is called {\em stable} (resp. {\em semistable}) if for any non-zero proper submodule $V'\subset V$ such that $V'(*D)\cong(\Phi^*)^{-1}(V')(*D)$ with the naturally induced parabolic structure, we have $\mu(V',\bullet)< (\text{resp. $\leq$}) \ \mu(V,\bullet)$.
\end{Definition}

In \cite{TM5}, Mochizuki built the following Kobayashi--Hitchin type correspondence between triply periodic monopoles and $a$-difference modules:

\begin{Theorem}[\cite{TM5}]For each $a\in\mathbb{C}/\Gamma_0$, there is an one to one correspondence between irreducible meromorphic triply periodic monopoles and stable parabolic $a$-difference modules of degree~$0$.
\end{Theorem}

\section{Conformal limits and Gaiotto's conjecture}

There are many interesting applications of the non-Abelian Hodge theory. Here we give a brief introduction as the beginning of this section. Throughout this section, $X$ will be denoted as a~compact Riemann surface of genus $g\geq2$. Let~$\mathcal{X}$ be the underlying smooth oriented surface, and denoted by~$K_X$ the canonical line bundle.

One important application of the non-Abelian Hodge theory is to study the theory for real Lie groups. Then non-Abelian Hodge theory is an important tool to find special connected components of the {\em Betti moduli space}\footnote{This moduli space is usually called {\em character variety} or {\em representation variety} by people work on surface group representation theory and higher Teichm\"uller theory, e.g., \cite{BC,Lab,Li,Wie}.} $M_{\mathrm{B}}(\mathcal{X},G):=\mathrm{Hom}(\pi_1(\mathcal{X}),G)/\!\!/G$ that parametrizes representations of $\pi_1(\mathcal{X})$ into the real Lie group $G$. We will give more details in Remark \ref{rmk3.1}.

Another important application of the non-Abelian Hodge theory is to study the relation between certain irreducible components\footnote{In fact, we will see in the next section that they are closed stata of the stratifications given by $\mathbb{C}^*$-action on the corresponding moduli spaces.} of $M_{\rm Dol}(X,r)$ and of $M_{\rm dR}(X,r)$. As holomorphic symplectic manifolds \cite{NH}, both $M_{\rm Dol}(X,r)$ and $M_{\rm dR}(X,r)$ have special holomorphic Lagrangian submanifolds that can be parametrized by holomorphic differentials on $X$. These are the {\em Hitchin section} in $M_{\rm Dol}$ (it is the image of the Hitchin base under a section of the Hitchin map) and the space of {\em opers} in $M_{\rm dR}$.

In this section, we will introduce a famous conjecture by Gaiotto that predicts the Hitchin section in the Dolbeault side and the space of opers in the de Rham side can be identified by a~map called {\em conformal limit}. Recently, this conjecture is confirmed by the authors of~\cite{DFKMMN}. In~\cite{CW}, by studying conformal limit, the authors generalize this result to any strata of the stratification of moduli spaces given by $\mathbb{C}^*$-action.

\subsection[$\mathbb{C}^*$-action on flat $\lambda$-connections]{$\boldsymbol{\mathbb{C}^*}$-action on flat $\boldsymbol{\lambda}$-connections}\label{sec3.1}

The Dolbeault moduli space $M_{\rm Dol}(X,r)$ admits a natural $\mathbb{C}^*$-action \cite{NH,CS3}:
\begin{gather*}
t\cdot\big[\big(E,\bar{\partial}_E,\theta\big)\big]:=\big[\big(E,\bar{\partial}_E,t\theta\big)\big].
\end{gather*}
This action is well-defined because it does not change the stability and vanishing of Chern classes. In \cite{CS6}, Simpson showed the existence of the coarse moduli space $M_{\rm Hod}(X,r)$ of semistable $\lambda$-flat bundles over $X$ of rank $r$ and with vanishing Chern classes. This moduli space admits a fibration $\pi\colon M_{\rm Hod}(X,r)\to\mathbb{C}$ such that $\pi^{-1}(1)=M_{\rm dR}(X,r)$ and $\pi^{-1}(0)=M_{\rm Dol}(X,r)$. Denoted by $M_{{\rm Hod}}^\lambda(X,r):=\pi^{-1}(\lambda)$ the fiber over $\lambda\in\mathbb{C}$.

Moreover, the $\mathbb{C}^*$-action on $M_{\rm Dol}(X,r)$ extends to an action of $\mathbb{C}^*$ on $M_{\rm Hod}(X,r)$:
\begin{gather*}
t\cdot\big[E,\bar{\partial}_E,D^\lambda,\lambda\big]:=\big[E,\bar{\partial}_E,tD^\lambda,t\lambda\big].
\end{gather*}
This action is well-defined and sends a flat $\lambda$-connection to a flat $t\lambda$-connection, so the fixed points must lie inside the fiber at $\lambda=0$, i.e., in $M_{\rm Dol}(X,r)$. By \cite[Lemma~4.1]{CS3}, the fixed points in $M_{\rm Dol}(X,r)$ are those Higgs bundles having the structure of systems of Hodge bundles,\footnote{In many contexts, these fixed points are called complex variations of Hodge structure.} that is, of the form
\begin{gather*}
\left[E=\bigoplus_{p=1}^kE^p,\, \bar{\partial}_E=\begin{pmatrix}
\bar{\partial}_{E^1}&&\\&\ddots&\\&&\bar{\partial}_{E^k}
\end{pmatrix}, \, \theta\colon E^p\to E^{p-1}\otimes\Omega_X^1\right].
\end{gather*}
Following the notations in~\cite{CS8}, let~$P$ be the subset consisting all the fixed points, and let $P=\coprod\limits_\alpha P_\alpha$ be the union of its connected components. This fixed point set will play an important role on the stratifications of moduli spaces, which will be described in the next section.

\subsection{Hitchin sections and opers}
Denote by $M_{\rm Dol}({\rm SL}(r,\mathbb{C}))$ the moduli space of semistable ${\rm SL}(r,\mathbb{C})$-Higgs bundles over $X$. There is a proper map, called the {\em Hitchin map}
\begin{align*}
h\colon \ M_{\rm Dol}({\rm SL}(r,\mathbb{C}))& \longrightarrow B:=\bigoplus_{i=2}^rH^0\big(X,K_X^i\big),\\
\big[E,\bar{\partial}_E,\theta\big]& \mapsto{\rm char}(\theta):=(q_2(\theta),\dots,q_r(\theta)),
\end{align*}
with each $q_i(\theta)=\operatorname{Tr}\big(\wedge^i\theta\big)$. This map defines a completely integrable system \cite{NH}. Fix a line bundle $L$ such that $L^r\cong K_X^{\frac{r(r-1)}{2}}$, there are~$r^{2g}$ choices of such $L$, and each choice makes the Hitchin base~$B$ parametrizes a section called the {\em Hitchin section}
\begin{gather*}
s_h\colon \ B\longrightarrow M_{\rm Dol}({\rm SL}(r,\mathbb{C})).
\end{gather*}
More explicitly, to a point $(q_2,\dots,q_r)\in B$ one associates a ${\rm SL}(r,\mathbb{C})$-Higgs bundle
\begin{gather*}
\left[
\big(E,\bar{\partial}_E\big)=L\oplus \big(L\otimes K_X^{-1}\big)\oplus\big(L\otimes K_X^{-2}\big)\oplus\cdots\oplus \big(L\otimes K_X^{-(r-1)}\big),\vphantom{\begin{pmatrix}
0&q_2&q_3&\cdots&q_r\\
a_1&0&q_2&\cdots&q_{r-1}\\
&a_2&0&\ddots&\vdots\\
&&\ddots&\ddots&q_2\\
&&&a_{r-1}&0\\
\end{pmatrix}}\right.\\
\left.\qquad
\theta=\begin{pmatrix}
0&q_2&q_3&\cdots&q_r\\
a_1&0&q_2&\cdots&q_{r-1}\\
&a_2&0&\ddots&\vdots\\
&&\ddots&\ddots&q_2\\
&&&a_{r-1}&0\\
\end{pmatrix}\right],
\end{gather*}
where each $a_i=\frac{i(r-i)}{2}$. For each choice of $L$, the image $s_h(B)$ of the Hitchin base~$B$ consists a~connected component of $M_{\rm Dol}({\rm SL}(r,\mathbb{C}))$, called a {\em Hitchin component}, and denoted as~${\rm Hit}_r$. This is related to special connected components of Betti moduli space (see the following remark).

\begin{Remark}\label{rmk3.1} The Betti moduli space $M_{\rm B}(\mathcal{X},{\rm PSL}(2,\mathbb{R}))\!:=\!\mathrm{Hom}(\pi_1(\mathcal{X}),{\rm PSL}(2,\mathbb{R}))/\!\!/{\rm PSL}(2,\mathbb{R})$ is the coarse moduli space of representations of $\pi_1(\mathcal{X})$ into ${\rm PSL}(2,\mathbb{R})$. In this Betti side, a Hitchin component usually means a connected component of $M_{\rm B}(\mathcal{X},{\rm PSL}(2,\mathbb{R}))$ that consists entirely the discrete and faithful representations $\rho\colon \pi_1(\mathcal{X})\to{\rm PSL}(2,\mathbb{R})$. In~\cite{Gol}, by using Euler number as invariant, Goldman showed that $M_{\rm B}(\mathcal{X},{\rm PSL}(2,\mathbb{R}))$ possesses $4g-3$ connected components with the corresponding Euler numbers from $2-2g$ to $2g-2$. Among these components it has two Hitchin components, as copies of the Teichm\"uller spaces $\mathrm{Teich}(\mathcal{X})$ and $\mathrm{Teich}\big(\bar{\mathcal{X}}\big)$, respectively. Moreover, each Hitchin component is homeomorphic to $\mathbb{R}^{6g-6}$.

In \cite{NH2}, Hitchin introduced the {\em $r$-Fuchsian representations} ($r>2$): representations $\rho\colon \pi_1(\mathcal{X})\allowbreak \to{\rm PSL}(r,\mathbb{R})$ that are the composition of a discrete and faithful representation $\rho_0\colon \pi_1(\mathcal{X})\allowbreak \to{\rm PSL}(2,\mathbb{R})$ and the unique irreducible representation $i\colon {\rm PSL}(2,\mathbb{R})\to {\rm PSL}(r,\mathbb{R})$. A connected component of $M_{\rm B}(\mathcal{X},{\rm PSL}(r,\mathbb{R}))$ that contains $r$-Fuchsian representations would be still called a~Hitchin component. By applying the Higgs bundle theory, Hitchin showed that \linebreak $M_{\rm B}(\mathcal{X},{\rm PSL}(r,\mathbb{R}))$ possesses~3 connected components when~$r$ is odd, and~6 connected components when $r$ is even. Among these components it has 1 Hitchin component when $r$ is odd, and 2 Hitchin components when $r$ is even. Moreover, each Hitchin component is homeomorphic to $\mathbb{R}^{(2g-2)(r^2-1)}$.

A further step in \cite{NH2} is the study of special connected components (i.e., Hitchin components) of $M_{\rm B}\big(\mathcal{X},G^r\big)$ for any split real form $G^r$ of a complex simple Lie group $G$. Hitchin showed this moduli space has Hitchin components homeomorphic to $\mathbb{R}^{(2g-2)\dim(G^r)}$.

A natural generalization of Hitchin's theory is to consider more general real non-compact semisimple Lie groups, then the study of some spacial components of the Betti moduli space leads to a new filed: {\em higher Teichm\"uller theory}. These special connected components are usually called {\em higher Teichm\"uller spaces}. For example, when the Lie groups are non-compact of Hermitian type, then higher Teichm\"uller spaces correspond to those connected components with maximal Toledo invariant. The corresponding representations are called {\em maximal representations}, they are showed to be Anosov by Labourie~\cite{Lab}. In particular, they are discrete and faithful. Recently, higher Teichm\"uller spaces are found for some special real Lie groups, e.g.,~\cite{BIW0,ABC}. Moreover, there are many good references on the introduction and the development of this theory from different points of view (e.g., \cite{DA,BIW,BC,OG,Wie}).
\end{Remark}

\begin{Example}When $r=2$, we fix a line bundle $L$ such that $L^2\cong K_X$, we will write it as $K_X^{\frac{1}{2}}$. Then the corresponding Higgs bundles in ${\rm Hit}_2$ have the form
\[
\left(\big(E,\bar{\partial}_E\big)=K_X^{\frac{1}{2}}\oplus K_X^{-\frac{1}{2}},\, \theta=\begin{pmatrix}
0&q_2\\1&0
\end{pmatrix}\right),
\]
 so the Hitchin section is parametrized by the quadratic differentials $q_2\in H^0\big(X,K_X^2\big)$. Moreover, for any~$q_2$, this Higgs bundle is stable, since the only $\theta$-invariant proper subbundle is $K_X^{-\frac{1}{2}}$. ${\rm Hit}_2$ describes the Teichm\"uller space ${\rm Teich}(\mathcal{X})$ in terms of Higgs bundles, and the quadratic differential~$q_2$ measures the non-conformality of the harmonic diffeomorphism $(\mathcal{X},g_0)\to(\mathcal{X},g)$~\cite{Li}.
\end{Example}

For any $r$, each Higgs bundle $\big[E,\bar{\partial}_E,\theta\big]$ in ${\rm Hit}_r$ is stable. Now let $g_t$ be a gauge transformation given by
\begin{gather*}
g_t=\begin{pmatrix}
t^{\frac{r-1}{2}}&&\\&\ddots&\\&&t^{-\frac{r-1}{2}}
\end{pmatrix},
\end{gather*}
then
\begin{align*}
\lim_{t\to0}g_t\cdot(t\cdot\theta)\cdot g_t^{-1}=\begin{pmatrix}
0&&&\\a_2&0&&\\&\ddots&\ddots&\\&&a_{r-1}&0
\end{pmatrix}.
\end{align*}
As the limit of the $\mathbb{C}^*$-action on $\big[E,\bar{\partial}_E,\theta\big]$, the Higgs bundle{\samepage
\begin{gather*}
\left(\big(E,\bar{\partial}_E\big)=L\oplus \big(L\otimes K_X^{-1}\big)\oplus\big(L\otimes K_X^{-2}\big)\oplus\cdots\oplus \big(L\otimes K_X^{-(r-1)}\big),\vphantom{\begin{pmatrix}
0&&&\\a_2&0&&\\&\ddots&\ddots&\\&&a_{r-1}&0
\end{pmatrix}}\right.\\
\left. \qquad{} \theta'=\begin{pmatrix}
0&&&\\a_2&0&&\\&\ddots&\ddots&\\&&a_{r-1}&0
\end{pmatrix}\right)
\end{gather*}
is stable. Therefore, $\big[E,\bar{\partial}_E,\theta\big]$ is also stable, since stability is an open condition.}

For each $L$, the Higgs bundle $s_h(0,\dots,0)$ corresponds to the image of the zero point $(0,\dots,0)\allowbreak \in B$ under the Hitchin section~$s_h$ is called a~{\em uniformizing Higgs bundle}. We also call it a~{\em Fuchsian point}, since the associated flat bundle corresponds to the Fuchsian representation that uniformizes the Riemann surface itself~\cite{NH2}. In fact, in~\cite{NH2}, Hitchin showed that for each choice of $L$ such that $L^r\cong K_X^{\frac{r(r-1)}{2}}$, Higgs bundles in the Hitchin section ${\rm Hit}_r$ have real monodromy representations
\begin{gather*}
\rho\colon \ \pi_1(X)\to {\rm SL}(r,\mathbb{R}),
\end{gather*}
that is, elements in ${\rm Hit}_r$ are ${\rm SL}(r,\mathbb{R})$-Higgs bundles.

Another object we will introduce is {\em oper}, as a special object in the de Rham moduli spa\-ce~$M_{\rm dR}$.

For any complex connected reductive Lie group $G$, the notion of $G$-oper was introduced and studied by Beilinson and Drinfeld in~\cite{BD}. Here for our study of the correspondence between Hitchin sections and opers, we just consider the case when $G={\rm SL}(r,\mathbb{C})$.

\begin{Definition}A {\em ${\rm SL}(r,\mathbb{C})$-oper} over~$X$ is a triple $(E,\nabla,F^\bullet)$, where $E$ is a holomorphic vector bundle of rank $r$ equipped with an isomorphism $\Lambda^rE\cong\mathcal{O}_X$, $\nabla\colon E\to E\otimes_{\mathcal{O}_X}K_X$ is a holomorphic flat connection, and $F^\bullet$ is a filtration given by holomorphic subbundles of~$E$:
\begin{gather*}
0=F^r\subseteq F^{r-1}\subseteq\cdots\subseteq F^1\subseteq F^0=E
\end{gather*}
satisfies the following three conditions:
\begin{itemize}\itemsep=0pt
\item[(1)] the filtration is of full flag, i.e., each graded term $E^i:=F^i/F^{i+1}$ is a line bundle;
\item[(2)] Griffiths transversality: $\nabla\colon F^i\to F^{i-1}\otimes_{\mathcal{O}_X}K_X$ for $i=1,\dots,r$;
\item[(3)] the induced map $\theta\colon E^i\to E^{i-1}\otimes_{\mathcal{O}_X}K_X$ is an isomorphism for $i=1,\dots,r$.
\end{itemize}
\end{Definition}

\subsection{Conformal limits and Gaiotto's conjecture}

As holomorphic Lagrangian submanifolds of moduli spaces, both Hitchin section and the space of opers are parametrized by the Hitchin base $\bigoplus\limits_{i=2}^rH^0\big(X,K_X^i\big)$ (details can be found in \cite[Section~2.7]{CW}). Moreover, they appear as closed subsets of the moduli spaces~\cite{BD,CW, NH}. The non-Abelian Hodge correspondence relates the moduli spaces~$M_{\rm Dol}$ and~$M_{\rm dR}$, however, it doesn't relate the Hitchin section and the space of opers. One would hope to find a map to relate these two Lagrangian submanifolds, and if possible, gives an identification between them. Gaiotto's conjecture arises from this consideration.

Let $\big[E,\bar{\partial}_E,\theta\big]\in M_{\rm Dol}(X,r)$ be a (poly)stable Higgs bundle with the pluri-harmonic metric denoted as $h$. Then for any $R>0$, the $\mathbb{C}^*$-action gives us a family of (poly)stable Higgs bundles $\big[E,\bar{\partial}_E,R\theta\big]$ with the corresponding pluri-harmonic metrics denoted as~$h_R$. Each $\big[E,\bar{\partial}_E,R\theta\big]$ determines a family of flat connections
\begin{gather*}
D_{R,\lambda}:=\bar{\partial}_E+\partial_{E,h_R}+\lambda^{-1}R\theta+\lambda R\theta^{\dagger}_{h_R}
\end{gather*}
parametrized by $\lambda\in\mathbb{C}^*$. If we fix the product $\hslash=\lambda R^{-1}$, then we can obtain a family of flat connections
\begin{gather*}
D_{R,\hslash}:=\bar{\partial}_E+\partial_{E,h_R}+\hslash^{-1}\theta+\hslash R^2\theta^{\dagger}_{h_R}
\end{gather*}
parametrized by $R\in\mathbb{R}^+$. The study of the limit $\lim\limits_{R\to0}D_{R,\hslash}$ (if exits) of the flat connection~$D_{R,\hslash}$ is an interesting question, this is the following definition:

\begin{Definition}
For a (poly)stable Higgs bundle $\big[E,\bar{\partial}_E,\theta\big]\in M_{\rm Dol}(X,r)$, the limit
\begin{gather*}
\lim_{R\to0}\big(E,\bar{\partial}_E+\hslash R^2\theta^{\dagger}_{h_R},\partial_{E,h_R}+\hslash^{-1}\theta\big)
\end{gather*}
if exists, is called its {\em conformal limit}.
\end{Definition}

We can interpret this flat connection in terms of $\lambda$-connections and its $\mathbb{C}^*$-action introduced in Section~\ref{sec3.1}. For each $R>0$, every stable Higgs bundle $\big[E,\bar{\partial}_E,\theta\big]$ determines a stable Higgs bundle $\big[E,\bar{\partial}_E,R\theta\big]$, thus a $\lambda$-flat bundle $
\big[E,\bar{\partial}_E+\lambda R\theta^{\dagger}_{h_R}, \lambda\partial_{E,h_R}+R\theta\big]$ by non-Abelian Hodge correspondence. In fact, this $\lambda$-flat bundle is obtained from $\lambda$ acts on the Higgs bundle $\big[E,\bar{\partial}_E+\lambda R\theta^{\dagger}_{h_R},\partial_{E,h_R}+\lambda^{-1}R\theta\big]$:
\begin{gather*}
\big[E,\bar{\partial}_E+\lambda R\theta^{\dagger}_{h_R},\lambda\partial_{E,h_R}+R\theta\big]=\lambda\cdot\big[E,\bar{\partial}_E+\lambda R\theta^{\dagger}_{h_R},\partial_{E,h_R}+\lambda^{-1}R\theta\big].
\end{gather*}
 Again by non-Abelian Hodge correspondence, this Higgs bundle corresponds to a flat bundle, its flat connection is exactly~$D_{R,\hslash}$.

For special choice of the conformal constant $\hslash$ and the initial Higgs bundle $\big[E,\bar{\partial}_E,\theta\big]$, its confromal limit recovers the non-Abelian Hodge map ${\rm NAH}$ defined in Section~\ref{sec2.1}. In fact, when $\hslash=1$ and $\big[E,\bar{\partial}_E,\theta\big]$ is a fixed point of the $\mathbb{C}^*$-action, then
\begin{align*}
\lim_{R\to0}D_{R,1}={\rm NAH}\big(\bar{\partial}_E,\theta\big).
\end{align*}

In \cite{Gai}, Gaiotto proposed the following conjecture that relates the Hitchin component and the locus of opers by the conformal limits:

\begin{Conjecture}For any Higgs bundle in the Hitchin component, its conformal limit exists and is an oper. Moreover, it gives a biholomorphism between the Hitchin component and the space of opers.
\end{Conjecture}

In the paper of \cite{DFKMMN}, the authors confirmed this conjecture as the following theorem:

\begin{Theorem}[\cite{DFKMMN}]\label{thm3}For any simple and simply connected Lie group~$G$. If $\big(E,\bar{\partial}_E,\theta\big)$ is a~$G$-Higgs bundle in the Hitchin component, then its conformal limit exists and lies in the space of opers. Moreover, it gives a biholomorphism between the Hitchin component and the space of opers.
\end{Theorem}

In fact, the locus ${\rm NAH}(s_h(B))$ of the Hitchin section under the non-Abelian Hodge correspondence in $M_{\rm dR}(G)$ intersects with the space ${\rm Op}(r)$ of opers transversely at the uniformizing point $s_h(0,\dots,0)$~\cite{LW}. Theorem~\ref{thm3} tells us the conformal limits give an identification between them.

In \cite{CW}, the authors studied the conformal limit for general Higgs bundle $\big[E,\bar{\partial}_E,\theta\big]\in M_{\rm Dol}$ such that the limiting point $\big[E,\bar{\partial}_0,\theta_0\big]:=\lim\limits_{t\to0}t\cdot[E,\bar{\partial}_E,\theta]$ is stable. They obtained the following general conformal limit correspondence:

\begin{Theorem}[\cite{CW}]\label{CW}
For any Higgs bundle $\big[E,\bar{\partial}_E,\theta\big]\in M_{\rm Dol}(\mathrm{SL}(r,\mathbb{C}))$ such that the limiting point $u:=\lim\limits_{t\to0}t\cdot\big[E,\bar{\partial}_E,\theta\big]\in P_\alpha$ is stable, its conformal limit always exists. Moreover, it gives a biholomorphism between $G_\alpha^0(u)$ and $G_\alpha^1(u)$.
\end{Theorem}

Here $G_\alpha^0$ and $G_\alpha^1$ stand for the strata in $M_{\rm Dol}$ and $M_{\rm dR}$ that correspond to the connected component $P_\alpha$ of the fixed point set. Details on this related to the stratifications of moduli spaces given by $\mathbb{C}^*$-action will be described in the next section.

This result generalizes Theorem~\ref{thm3} since the limiting point of any element in the Hitchin section is stable.

\section{Stratification of moduli spaces}

Throughout this section, $X$ denotes a compact Riemann surface.

Following the description in Section~\ref{sec3.1}, the natural $\mathbb{C}^*$-action on $M_{\rm Dol}(X,r)$ extends to an action of $\mathbb{C}^*$ on $M_{\rm Hod}(X,r)$. By a work of Simpson \cite{CS8}, for each $\big[E,\bar{\partial}_E,D^\lambda\big]\in M_{\rm Hod}(X,r)$, the limit $\lim\limits_{t\to0}t\cdot\big[E,\bar{\partial}_E,D^\lambda\big]$ exists and as a fixed point lies in $P$. Therefore, this action gives a~Bia{\l}ynicki-Birula stratification of~$M_{\rm Hod}$ into locally closed subsets (see Section~\ref{sec4.2}). Restricting the stratification to the fiber over~$0$, it recovers the classical Bia{\l}ynicki-Birula stratification of~$M_{\rm Dol}$. Restricting the stratification to the fiber over~$1$, we will have a stratification of~$M_{\rm dR}$ into locally closed subsets, and moreover, the space of opers, appears as a special stratum. This stratification is new to us, provides a new direction on the study of $M_{\rm dR}$.

On the other hand, in the same paper, Simpson showed each flat bundle over $X$ admits a~filtration satisfies the Griffiths transversality and such that the induced graded Higgs bundle is semistable (see Section \ref{sec4.1}). And moreover, the limit of the $\mathbb{C}^*$-action on that flat bundle, as a fixed point in $M_{\rm Dol}$, is S-equivalent to the graded Higgs bundle. This property provides a~possibility on the description of certain flat bundles (e.g., Theorem~\ref{thm4.10}), and moreover, plays an important role on the study of~$M_{\rm dR}$ (e.g., on the proof of Conjecture~\ref{nest} for rank~2).

This section can be treated as an application of non-Abelian Hodge theory to the study of the de Rham moduli space $M_{\rm dR}$. In this section, we will describe the stratifications of moduli spaces given by $\mathbb{C}^*$-action. Meanwhile, some conjectures related to the study of~$M_{\rm dR}$ will be introduced.

\subsection{Simpson filtrations on flat bundles}\label{sec4.1}

For the $\mathbb{C}^*$-action on the Dolbeault moduli space $M_{\rm Dol}(X,r)$. Since the Hitchin map $h\colon \! M_{\rm Dol}(X,r)\!\!\allowbreak \to\bigoplus\limits_{i=1}^rH^0\big(X,K_X^i\big)$ is proper and $\mathbb{C}^*$-equivariant, for any $\big[E,\bar{\partial}_E,\theta\big]\in M_{\rm Dol}(X,r)$, the limit $\lim\limits_{t\to0}t\cdot\big[E,\bar{\partial}_E,\theta\big]$ exists and as a fixed point of this action.

There is no analogue of the Hitchin map for $M_{\rm Hod}(X,r)$. However, for each $\big[E,\bar{\partial}_E,D^\lambda,\lambda\big]\in M_{\rm Hod}(X,r)$, the limit $\lim\limits_{t\to0}t\cdot\big[E,\bar{\partial}_E,D^\lambda,\lambda\big]$ of the $\mathbb{C}^*$-action still exists as a fixed point and lies in some~$P_\alpha$~\cite{CS8}. In particular, the limit $\lim\limits_{t\to0}t\cdot(E,\nabla)$ of a flat bundle~$(E,\nabla)$ exists as a fixed point. Moreover, this limit can be described by the existence of a special filtration of this flat bundle. This filtration, is found by Simpson, we will call it a~{\em Simpson filtration} throughout the whole paper.

\begin{Definition}[\cite{H,CS8}]
Let $E$ be a vector bundle over $X$ with flat connection $\nabla\colon E\rightarrow E\otimes_{\mathcal{O}_X}\Omega^1_X$. A decreasing filtration~$F^\bullet$ of~$E$ by strict subbundles
\begin{gather*}
E=F^0\supset F^1\supset\cdots\supset F^k=0
\end{gather*}
is called a \emph{Simpson filtration} if it satisfies the following two conditions:
\begin{enumerate}\itemsep=0pt
 \item[(1)] Griffiths transversality: $\nabla\colon F^p\rightarrow F^{p-1}\otimes_{\mathcal{O}_X}\Omega^1_X$ for $p=1,\dots,k$;
 \item[(2)] graded-semistability: the associated graded Higgs bundle $(\mathrm{Gr}_F(E),\mathrm{Gr}_F(\nabla))$, where \linebreak $\mathrm{Gr}_F(E) =\bigoplus\limits_{p=0}^{k-1}E^p$ with $E^p=F^p/F^{p+1}$ and $\mathrm{Gr}_F(\nabla)\colon E^p\rightarrow E^{p-1}\otimes_{\mathcal{O}_X}\Omega^1_X$ induced from~$\nabla$, is a~semistable Higgs bundle.
\end{enumerate}
Such a triple $\big(E,\nabla,F^\bullet\big)$ is called a {\em partial oper}.
\end{Definition}

Simpson proved the following nice theorem~\cite{CS8}.

\begin{Theorem}\label{thm4.3}Let $(E,\nabla)$ be a flat bundle over a smooth projective curve $X$.
\begin{enumerate}\itemsep=0pt
\item[$(1)$] There exist Simpson filtrations $F^\bullet$ on $(E,\nabla)$, this means any flat bundle has partial oper structure.
\item[$(2)$] Let $F_1^\bullet$ and $F_2^\bullet$ be two Simpson filtrations on $(E,\nabla)$, then the associated graded Higgs bundles $(\mathrm{Gr}_{F_1}(E),\mathrm{Gr}_{F_1}(\nabla))$ and $(\mathrm{Gr}_{F_2}(E),\mathrm{Gr}_{F_2}(\nabla))$ are $S$-equivalent.
\item[$(3)$] $(E,\nabla,F^\bullet)$ is graded-stable if and only if the Simpson filtration is unique $($up to indices translation$)$.
\item[$(4)$] $\lim\limits_{t\rightarrow 0}t\cdot(E,\nabla)=[\mathrm{Gr}_F(E),\mathrm{Gr}_F(\nabla)]\in M_{\rm Dol}(X,r)$.
\end{enumerate}
\end{Theorem}

In \cite{CS8}, Simpson gave a wonderful iterated process to show the existence of Simpson filtration for $(E,\nabla)$. We now sketch how it works.

Suppose $(E,\nabla)$ admits a filtration
\begin{gather*}
F^\bullet\colon \ 0\subset F^{k-1}\subset\cdots\subset F^0=E
\end{gather*}
that satisfies the Griffiths transversality $\nabla(F^p)\subset F^{p-1}\otimes\Omega_X^1$, and such that the associated Higgs bundle $(V,\theta):=({\rm Gr}_F(E),{\rm Gr}_F(\nabla))$ is not semistable. To see the existence of such filtration, we can begin with the trivial filtration $0\subset F^0=E$, the graded Higgs bundle will be $({\rm Gr}_F(E),{\rm Gr}_F(\nabla))=(E,0)$. Then applying the following iteration process, we can always have a such filtration. Take $H\subset(V,\theta)$ to be the maximal destabilizing subsheaf, which is known being unique and a subbundle of $V$, and the quotient $V/H$ is also a subbundle of $E$. Since $H$ is unique, it must be a fixed point of the $\mathbb{C}^*$-action. Therefore, $H$ has a structure of system of Hodge bundles, and as a sub-system of Hodge bundles of $(V,\theta)$, that is, $H=\bigoplus H^p$ with each $H^p=H\cap {\rm Gr}_F^p(E)\subset F^p(E)/F^{p+1}(E)$ being a strict subbundle.

The new filtration $G^\bullet$ is defined as
\begin{gather*}
G^p:=\operatorname{Ker}\left(E\to\frac{E/F^p(E)}{H^{p-1}}\right).
\end{gather*}
It satisfies the Griffiths traversality since $\theta(H^p)\subset H^{p-1}\otimes\Omega_X^1$, and it fits into the exact sequence
\begin{gather*}
0\to {\rm Gr}_F^p(E)/H^p\to{\rm Gr}_G^p(E)\to H^{p-1}\to0.
\end{gather*}
If the new resulting graded Higgs bundle $({\rm Gr}_G(E),{\rm Gr}_G(\nabla))$ is still not semistable, then we continue this process to obtain a new graded Higgs bundle. By introducing three bounded invariants, Simpson showed that the iteration process will strictly decrease these invariants in lexicographic order. Therefore, after a finite step, we will find a filtration such that the associated graded Higgs bundle is semistable.

\begin{Corollary}[\cite{CS8}]\label{lem4.4}Let $X$ be a smooth projective curve and let $\big[E,\bar{\partial}_E,D^\lambda\big]\in M_{\rm Hod}(X,r)$ be any $\lambda$-flat bundle $(\lambda\neq0)$ in the moduli space, then
\begin{gather*}
\lim_{t\to0}t\cdot\big[E,\bar{\partial}_E,D^\lambda\big]=\lim_{t\to0}t\cdot \big[E,\bar{\partial}_E,\lambda^{-1}D^\lambda\big]\in M_{\rm Dol}(X,r),
\end{gather*}
where $\big(E,\bar{\partial}_E,\lambda^{-\lambda}D^\lambda\big)$ is the flat bundle $(1$-bundle$)$ associated to $\big(E,\bar{\partial}_E,D^\lambda\big)$.
\end{Corollary}

By definition, the filtration for a $\mathrm{GL}(r,\mathbb{C})$-oper is a special Simpson filtration.

\begin{Corollary}
Every oper $(E,\nabla,F^\bullet)$ over a smooth projective curve $X$ of $g\geq2$ is graded stable, in particular, $(E,\nabla)$ has $F^\bullet$ as the only Simpson filtration $($up to indices translation$)$.
\end{Corollary}

\begin{proof}Let $\big(\bigoplus\limits_{i=0}^{r-1}E^i,\theta\big)$ be the associated graded Higgs bundle, where each $E^i$ is a line bundle and each $\theta|_{E^i}\colon E^i\to E^{i-1}\otimes\Omega_X^1$ is an isomorphism. This means
\begin{gather*}
\deg\big(E^i\big)=\deg\big(E^{i-1}\big)+2g-2=\cdots=\deg\big(E^0\big)+i(2g-2),
\end{gather*}
note that each $\theta$-invariant non-zero proper subbundle of $\big(\bigoplus\limits_{i=0}^{r-1}E^i,\theta\big)$ has the form $\bigoplus\limits_{i=0}^kE^i$ ($0\leq k<r-1$), with
\begin{gather*}
\deg\left(\bigoplus_{i=0}^kE^k\right)=(k+1)\deg\big(E^0\big)+k(k+1)(g-1),
\end{gather*}
so
\begin{gather*}
\mu\left(\bigoplus_{i=1}^kE^i\right)=\deg\big(E^0\big)+k(g-1)<\deg\big(E^0\big)+(r-1)(g-1)=\mu(E),
\end{gather*}
this means $(E,\nabla,F^\bullet)$ is graded stable. In particular, by~(3) of Theorem \ref{thm4.3}, $F^\bullet$ is the only Simpson filtration for~$(E,\nabla)$.
\end{proof}

\begin{Remark}The non-uniqueness of the Simpson filtration is easy to see. In fact, any irreducible rank 2 flat bundle of degree 0 with the underlying vector bundle strictly semistable which is an extension of a degree 0 line bundle admits more than one Simpson filtration. One is the trivial filtration, and the other one has two terms with the first term the extension of line bundle. The two resulting graded Higgs bundles are automatically semistable and $S$-equivalent to each other, as a unique representative point in the Dolbeault moduli space, which parametrizes the limit point of the $\mathbb{C}^*$-action.
\end{Remark}

\subsection{Stratifications of moduli spaces}\label{sec4.2}

Following \cite{CS8}, we introduce the following set:
\begin{gather*}
G_\alpha:=\big\{ \big[E,\bar{\partial}_E,D^\lambda,\lambda\big]\in M_{\rm Hod}(X,r) \, \big|\, \lim_{t\to0}t\cdot \big[E,\bar{\partial}_E,D^\lambda,\lambda\big]\in P_\alpha\big\},
\end{gather*}
then these $G_\alpha$ gives a Bia{\l}ynicki-Birula type stratification of the Hodge moduli space $M_{\rm Hod}(X,r)$
\begin{gather*}
M_{\rm Hod}(X,r)=\bigcup_\alpha G_\alpha
\end{gather*}
into locally closed subsets. There is a natural projection $p_\alpha\colon G_\alpha\to P_\alpha$ by taking the limit of the~$\mathbb{C}^*$-action. Restricting the stratification to the fiber over each $\lambda\in\mathbb{C}$ gives the stratification of $M_{{\rm Hod}}^\lambda(X,r)(X,r)$
\begin{gather*}
M_{{\rm Hod}}^\lambda(X,r)(X,r)=\bigcup_\alpha G_\alpha^\lambda:=\bigcup_\alpha\Big(G_\alpha\bigcap\pi^{-1}(\lambda)\Big)
\end{gather*}
into locally closed subsets. In particular, taking $\lambda=0$ and $1$, we have the stratifications of $M_{\rm Dol}(X,r)$ and $M_{\rm dR}(X,r)$
\begin{gather*}
M_{\rm Dol}(X,r)=\bigcup_\alpha G_\alpha^0, \qquad
M_{\rm dR}(X,r)=\bigcup_\alpha G_\alpha^1
\end{gather*}
into locally closed subsets. The first one is in fact the Bia{\l}ynicki-Birula stratification of \linebreak $M_{\rm Dol}(X,r)$ given by the $\mathbb{C}^*$-action. The second one is called the {\em oper stratification} of $M_{\rm dR}(X,r)$, since the space of opers appears as a special stratum (Theorem~\ref{thm4.3} and Corollary~\ref{lem4.4}). The projection~$p_\alpha$ restricts on fibers gives projections $p_\alpha^0\colon G_\alpha^0\to P_\alpha$ and $p_\alpha^1\colon G_\alpha^1\to P_\alpha$. By Bia{\l}ynicki-Birula theo\-ry, over each point $u\in P_\alpha$, the fiber $G_\alpha^0(u):=\big(p_\alpha^0\big)^{-1}(u)$ is an affine space. Moreover, in~\cite{CW}, by applying the conformal limit techniques, the authors showed each $G_\alpha^1(u):=\big(p_\alpha^1\big)^{-1}(u)$ is also affine \cite[Corollary~1.5]{CW}.

In \cite{CS8}, Simpson showed that over any point $u$ which is a stable system of Hodge bundles in the fixed point set, the fiber~$G_\alpha^1(u)$ is a Lagrangian submanifold of~$M_{\rm dR}$. In~\cite{CW}, the authors showed the Lagrangian property also holds for $M_{\rm Dol}(X,r)$, that is, for each stable~$u$, the fiber $G_\alpha^0(u)$ is a Lagrangian submanifold of~$M_{\rm Dol}$. Moreover, by studying the conformal limits, they showed each fiber $G_\alpha^0(u)$ is in fact biholomorphic to $G_\alpha^1(u)$ (see Theorem \ref{CW}), which generalizes the conformal limit correspondence between Hitchin section and the space of opers. As we have seen in the last section that the locus ${\rm NAH}(s_h(B))$ intersects with ${\rm Op}(r)$ transversely at $s_h(0,\dots,0)$, the authors of \cite{CW} showed that under the non-Abelian Hodge correspondence, each fiber of the Bia{\l}ynicki-Birula stratification and intersects with the fiber of the oper stratification transversely at the base point. More explicitly, for each stable point $u\in P_\alpha$, the image ${\rm NAH}\big(G_\alpha^0(u)\big)$ of the fiber~$G_\alpha^0(u)$ under non-Abelian Hodge correspondence intersects with the fiber~$G_\alpha^1(u)$ transversely at~${\rm NAH}(u)$.

The non-Abelian Hodge correspondence shares no light on the study of~$M_{\rm dR}$ from the study of~$M_{\rm Dol}$, since~$M_{\rm dR}$ and~$M_{\rm Dol}$ share very few similarities as algebraic spaces. However, the strata~$G_\alpha^1$ play an important role on the understanding of~$M_{\rm dR}$. Especially the Lagrangian property for each fiber of $G_\alpha^1$ induces a natural question on the relationship between these Lagrangian fibers:

\begin{Conjecture}[foliation conjecture, \cite{CS8}] When varying~$\alpha$, these Lagrangian fibers of $p_\alpha^1\colon G_\alpha^1\allowbreak \to P_\alpha$ fit together to provide a smooth foliation of $M_{\rm dR}(X,r)$ with each leaf closed.
\end{Conjecture}

This conjecture is still open, one progress was recently made by the authors of~\cite{LSS} for the case of moduli space of rank~2 parabolic connections on~$\mathbb{P}^1$ minus~4 points.

This closedness property for $M_{\rm Dol}(X,r)$ is clearly not right, since any fiber contained in the compact nilpotent cone would be not closed. If fact, if it is closed, as a subset of a compact space, it is compact also, as it is affine, this couldn't happen.

We give a more explicit explanation here, define the following indexed sets by the limit of $\mathbb{C}^*$-action:
\begin{gather*}
D_\alpha^0:=\big\{ \big[E,\bar{\partial}_E,\theta\big]\in M_{\rm Dol}(X,r) \,| \, \lim_{t\to\infty}t\cdot\big[E,\bar{\partial}_E,\theta\big]\in P_\alpha\big\}.
\end{gather*}
Then by Hausel's thesis~\cite{TH}, these sets fit together into the nilpotent cone:
\begin{align}\label{2}
h^{-1}(0)=\bigcup_\alpha D_\alpha^0,
\end{align}
which is a deformation retraction of the whole moduli space. Let $u\in P_\alpha$ be a fixed point such that the whole fiber $G_\alpha^0(u)=\big\{[E,\bar{\partial}_E,\theta]\,|\, \lim\limits_{t\to0}t\cdot\big[E,\bar{\partial}_E,\theta\big]=u\big\}$ is contained in the nilpotent cone, that is, $G_\alpha^0(u)\subseteq h^{-1}(0)$. Take any $v\in G_\alpha^0(u)$ that is not a fixed point, as $t\cdot v\in G_\alpha^0(u)$ for all $t\in\mathbb{C}^*$ and $G_\alpha^0(u)$ is closed, both $\lim\limits_{t\to0}t\cdot v$ and $\lim\limits_{t\to\infty}t\cdot v$ lie in $G_\alpha^0(u)$. Note the first limit is the fixed point $u$, and the second limit is also a fixed point by~\eqref{2}. By definition, $G_\alpha^0(u)$ can contain only one fixed point, $u$, this means the two limits of a non-fixed point should coincide, this couldn't happen since the fixed point sets are ordered by the energy functional.

Following this idea, with a discussion with Simpson, he told me the following pure algebraic-geometric result:

\begin{Lemma}Let $\mathbb{G}_m$ the multiplicative group, and $Y$ be any algebraic variety. Suppose $\mathbb{G}_m$ acts on $Y$ with open dense orbit isomorphic to $\mathbb{G}_m$, and such that the two endpoints are identified. Then there does not exist an ample linearized line bundle on $Y$.
\end{Lemma}

\begin{proof}Let $y$ be any point lies in the open dense orbit and let $z$ be an extra point not in the orbit, denote by $\mathcal{O}_y$ the orbit. Then by assumption, $\mathcal{O}_y$ and $\{z\}$ are the only two orbits of this action, and moreover, $\lim\limits_{t\to0}t\cdot y=\lim\limits_{t\to\infty}t\cdot y=z$. Suppose $Y$ has an ample $\mathbb{G}_m$-linearized line bundle $L$, then there exists an invariant section of $L^{\otimes n}$ for some $n\in\mathbb{Z}_+$. Hence the linear action of $\mathbb{G}_m$ on the fiber $L_{\lim\limits_{t\to0}t\cdot y}$ has positive weight, while it acts on the fiber $L_{\lim\limits_{t\to\infty}t\cdot y}$ has negative weight, which is a contradiction.
\end{proof}

With this lemma, we can see in another way why $G_\alpha^0(u)$ could not have two endpoints of the $\mathbb{C}^*$-action on a non-fixed point identified. Since the Dolbeault moduli space~$M_{\rm Dol}$ has an ample $\mathbb{G}_m$-linearized line bundle (see Simpson's construction of the moduli space~\cite{CS4,CS5}), it could not have such $G_\alpha^0(u)$ inside with the property. Therefore, $G_\alpha^0(u)$ could not have a~point with two endpoints of the $\mathbb{C}^*$-action identified, this means $G_\alpha^0(u)$ could not be closed.

Let $U(X,r)$ be the moduli space of semistable vector bundles over~$X$ of degree~0 and rank~$r$, this is known to be an irreducible variety. It naturally embeds into $M_{\rm Dol}(X,r)$ as Higgs bundles with zero Higgs field. Moreover, it appears as a connected component of the fixed point set~$P$, let~$P_0$ denotes this component. On the de Rham side, the corresponding stratum $G_0^1$ in $M_{\rm dR}$ is the unique open stratum that consists of flat bundles of the form $\big(E,\bar{\partial}_E,\partial+\varphi\big)$, where $\big(E,\bar{\partial}_E\big)$ is a polystable vector bundle, $\varphi\in H^0\big(X,\operatorname{End}(E)\otimes\Omega_X^1\big)$, and $\partial$ is the unique unitary flat connection. For each such flat bundle, the Simpson filtration is trivial. On the Dolbeault side, the corresponding stratum~$G_0^0$ is a~dense open subset of $M_{\rm Dol}(X,r)$, and can be identified with the cotangent bundle $T^*U(X,r)$. The strata $G_0^1$ and $G_0^0$ are usually called the {\em lowest strata}. And if we take $u=\big[E,\bar{\partial}_E,0\big]\in P_0$, then $G_0^0(u)=H^0 \big(X,\operatorname{End}(E)\otimes\Omega_X^1\big)\subseteq M_{\rm Dol}(X,r)$, the space of Higgs fields on~$E$. This fiber is closed if and only if~$E$ is very stable, i.e., there is no non-zero nilpotent Higgs field on~$E$ (see also~\cite{P}).

The space $P_u$ of uniformizing Higgs bundles (see last section) corresponds to the stratum of opers $G_u^1$ in $M_{\rm dR}$ and the stratum $G_u^0$ of Hitchin component in $M_{\rm Dol}$, moreover, $G_u^1$ is closed in $M_{\rm dR}$, since $G_u^0$ is closed in $M_{\rm Dol}$ \cite{CS8}. We will call them the {\em oper stratum} and the {\em Hitchin stratum}, respectively.

In \cite{CS8}, Simpson proposed another method to study the behaviour of the stratifications. Let~$M$ be a (quasi-)projective variety with a stratification of locally closed subsets $M=\coprod\limits_{\alpha\in\Lambda} G_\alpha$, we call this stratification {\em nested} if there is a partial order $(\Lambda,\leq)$ such that
\begin{gather*}
\overline{G_\alpha}=\coprod_{\beta\leq\alpha}G_\beta,
\end{gather*}
this implies the partial order is defined as
\begin{gather*}
\beta\leq\alpha \Longleftrightarrow G_\beta\subseteq\overline{G_\alpha}.
\end{gather*}

\begin{Conjecture}[nestedness conjecture]\label{nest} The stratifications for $M_{\rm Dol}(X,r)$ and $M_{\rm dR}(X,r)$ are both nested, and the arrangements for both stratifications are the same.
\end{Conjecture}

Simpson himself studied and showed it for rank 2 case by using the beautiful techniques of deformation theory, but for the higher rank case, it is still an open problem.

At least from the proof of nestedness conjecture of rank~2 case, we can see that the oper stratum (the highest stratum)~$G_u^1$ is the stratum of minimal dimension among all the strata, based on this, Simpson proposed another conjecture~\cite{CS8}:

\begin{Conjecture}[oper stratum conjecture] The oper stratum $G_u^1$ is the unique closed stratum and the unique stratum with minimal dimension.
\end{Conjecture}

In \cite{H}, by classifying irreducible components of the fixed point set $P$ of the $\mathbb{C}^*$-action on~$M_{\rm Dol}(X,r)$, we partially confirmed this conjecture:

\begin{Theorem}[\cite{H}]The oper stratum is the unique closed stratum with minimal dimension.
\end{Theorem}

We still do not know if there exists a closed stratum of strictly higher dimension. Moreover, we showed the Hitchin stratum is also the unique closed stratum of minimal dimension, but it is not the unique closed stratum.

In \cite{GZ}, the authors considered the relation between the Bia{\l}ynicki-Birula stratification and the Shatz stratification of~$M_{\rm Dol}(X,3)$, where the Shatz stratification is given by Harder--Narasimhan type of the underlying vector bundles of the Higgs bundles. With this inspiration, we can consider the relation between the oper stratification given by Simpson filtrations and the Shatz stratification of $M_{\rm dR}(X,3)$ given by Harder--Narasimhan type of the underlying vector bundles.

For a flat bundle $(E,\nabla)$ of rank~3, suppose $E$ is not a stable vector bundle. Then the Harder--Narasimhan type of $E$ can be given as follows:
\begin{enumerate}\itemsep=0pt
\item Type $(1,2)$, that is, the Harder--Narasimhan filtration is given by
$0\subsetneq H^1\subsetneq E$
with $\operatorname{rk}\big(H^1\big)=1$ and $\deg\big(H^1\big)=d_1$. In this case, $H_1\subseteq E$ is the maximal destabilizing subsheaf, so $d_1>0$.
\item Type $(2,1)$, that is, the Harder--Narasimhan filtration is given by
$0\subsetneq H^1\subsetneq E$
with $\operatorname{rk}\big(H^1\big)=2$ and $\deg\big(H^1\big)=d_1$. As in~(1), $H_1\subseteq E$ is the maximal destabilizing subsheaf, so $d_1>0$.
\item Type $(1,1,1)$, that is, the Harder--Narasimhan filtration is given by $0\subsetneq H^1\subsetneq H^2\subsetneq E$
with $\operatorname{rk}\big(H^1\big)=1$, $\deg\big(H^1\big)=d_1$ and $\operatorname{rk}\big(H^2\big)=2$, $\deg\big(H^2/H^1\big)=d_2$. In this case, $H_1\subseteq E$ is the maximal destabilizing subsheaf, and $H_2/H_1\subseteq E/H_1$ is the maximal destabilizing subsheaf, hence $d_1>0$ and $d_1+d_2>0$.
\end{enumerate}

\begin{Theorem}\label{thm4.10}Let $(E,\nabla)\in M_{\rm dR}(X,3)$ be an irreducible flat bundle of rank~$3$ that is graded-stable, and such that~$E$ is not a stable vector bundle. Then its Simpson filtration is determined by its Harder--Narasimhan filtration as follows:
\begin{itemize}\itemsep=0pt
\item[$(1)$] If the Harder--Narasimhan type of $E$ is $(1,2)$ as above. Let $I\subset E/H^1$ be the sub line bundle by saturating the subsheaf $\theta\big(H^1\big)\otimes K_X^{-1}\subset E/H^1$, where $\theta\colon H^1\to E/H^1\otimes K_X$ is the non-zero map induced by~$\nabla$, then
\begin{itemize}\itemsep=0pt
\item[$(1.1)$] $0<d_1<g-1$ $\&$ $d_1-2g+2\leq\deg(I)<-d_1$, the Simpson filtration coincides with the Harder--Narasimhan filtration. Hence
\begin{gather*}
\lim_{t\to0}t\cdot(E,\nabla)=\big[H^1\oplus E/H^1,\theta\big].
\end{gather*}
\item[$(1.2)$] $0<d_1<g-1$ $\&$ $-d_1<\deg(I)\leq-\frac{d_1}{2}$ or $g-1<d_1\leq\frac{4g-4}{3}$ $\&$ $d_1-2g+2\leq\deg(I)\leq-\frac{d_1}{2}$, in either case, the Simpson filtration is given by
\begin{gather*}
0\subsetneq H^1\subsetneq F^1\subsetneq E,
\end{gather*}
where $F^1=\operatorname{Ker}\big(E\to\frac{E/H^1}{I}\big)\subset E$ is a rank~$2$ subbundle. Hence
\begin{gather*}
\lim_{t\to0}t\cdot(E,\nabla)=\left[H^1\oplus I\oplus\frac{E/H^1}{I},\begin{pmatrix}
0&0&0\\ \varphi_1&0&0\\0&\varphi_2&0
\end{pmatrix}\right],
\end{gather*}
where $\varphi_1\colon H^1\to I\otimes K_X$ is induced by $\theta$ and $\varphi_2\colon I\to\frac{E/H^1}{I}\otimes K_X$ is induced by $\nabla\colon F^1\to E\otimes K_X$.
\end{itemize}
\item[$(2)$] If the Harder--Narasimhan type of $E$ is $(2,1)$ as above. Let $\theta\colon H^1\to E/H^1\otimes K_X$ be the non-zero induced map and let $N:=\operatorname{Ker}\big(H^1\to E/H^1\otimes K_X\big)\subset H^1$ be the sub line bundle, then
\begin{itemize}\itemsep=0pt
\item[$(2.1)$] $0<d_1<g-1$ $\&$ $2d_1-2g+2\leq\deg(N)<0$, the Simpson filtration coincides with the Harder--Narasimhan filtration. Hence
\begin{gather*}
\lim_{t\to0}t\cdot(E,\nabla)=\big[H^1\oplus E/H^1,\theta\big].
\end{gather*}
\item[$(2.2)$] $0<d_1<g-1$ $\&$ $0<\deg(N)\leq\frac{d_1}{2}$ or $g-1<d_1\leq\frac{4g-4}{3}$ $\&$ $2d_1-2g+2\leq\deg(N)\leq\frac{d_1}{2}$, in either case, the Simpson filtration is given by
\begin{gather*}
0\subsetneq N\subsetneq G^1\subsetneq E,
\end{gather*}
where $G_1:=\operatorname{Ker}\big(E\to E/H^1\big)\subset E$ is a rank~$2$ subbundle. Hence
\begin{gather*}
\lim_{t\to0}t\cdot(E,\nabla)=\left[N\oplus H^1/N\oplus E/H^1,\begin{pmatrix}
0&0&0\\ \psi_1&0&0\\0&\psi_2&0
\end{pmatrix}\right],
\end{gather*}
where $\psi_1\colon N\to H^1/N\otimes K_1$ and $\psi_2\colon H^1/N\to E/H^1\otimes K_X$ are induced by $\nabla\colon N\to H^1\otimes K_X$.
\end{itemize}
\item[$(3)$] If the Harder--Narasimhan type of $E$ is $(1,1,1)$ as above. Let $I\subset E/H^1$ be the sub line bundle as defined in~$(1)$ and $N\subset H^2$ be the sub line bundle as defined in~$(2)$.
\begin{itemize}\itemsep=0pt
\item[$(3.1)$] $\max\{-d_1,2d_2-d_1\}<\deg(I)\leq d_2$, the Simpson filtration is given by
\begin{gather*}
0\subsetneq H^1\subsetneq F^1\subsetneq E,
\end{gather*}
where $F^1=\operatorname{Ker}\big(E\to\frac{E/H^1}{I}\big)\subset E$ is a rank~$2$ subbundle. Hence
\begin{gather*}
\lim_{t\to0}t\cdot(E,\nabla)=\left[H^1\oplus I\oplus\frac{E/H^1}{I},\begin{pmatrix}
0&0&0\\ \varphi_1&0&0\\0&\varphi_2&0
\end{pmatrix}\right],
\end{gather*}
where $\varphi_1\colon H^1\to I\otimes K_X$ is induced by $\theta$ in~$(1)$ and $\varphi_2\colon I\to\frac{E/H^1}{I}\otimes K_X$ is induced by $\nabla\colon F^1\to E\otimes K_X$. In particular, if $\deg(I)=d_2$, then $F^1=H^2$, that is, the Simpson filtration coincides with the Hardar--Narasimhan filtration.
\item[$(3.2)$] $d_1-2g+2\leq\deg(I)<\min\{-d_1,2d_2-d_1\}$,
\begin{itemize}\itemsep=0pt
\item[$(3.2.1)$] $d_2<0$,
then the Simpson filtration is given by
\begin{gather*}
0\subsetneq H^1\subsetneq E.
\end{gather*}
Hence
\begin{gather*}
\lim_{t\to0}t\cdot(E,\nabla)=\big[H^1\oplus E/H^1,\theta\big],
\end{gather*}
where $\theta\colon H^1\to E/H^1\otimes K_X$ is induced from $\nabla\colon H^1\to E\otimes K_X$.
\item[$(3.2.2)$] $d_2>0$,
\begin{itemize}\itemsep=0pt
\item[$\bullet$] $2(d_1+d_2)-2g+2\leq\deg(N)<0$, the Simpson filtration is given by
\begin{gather*}
0\subsetneq H^2\subsetneq E.
\end{gather*}
Hence
\begin{gather*}
\lim_{t\to0}t\cdot(E,\nabla)=\big[H^2\oplus E/H^2,\theta\big],
\end{gather*}
where $\theta\colon H^2\to E/H^2\otimes K_X$ is induced from $\nabla\colon H^2\to E\otimes K_X$.
\item[$\bullet$] $0<\deg(N)\leq d_1$, the Simpson filtration is given by
\begin{gather*}
0\subsetneq N\subsetneq H^2\subsetneq E.
\end{gather*}
Hence
\begin{gather*}
\lim_{t\to0}t\cdot(E,\nabla)=\left[N\oplus H^2/N\oplus E/H^2,\begin{pmatrix}
0&0&0\\ \varphi_1&0&0\\0&\varphi_2&0
\end{pmatrix}\right],
\end{gather*}
where $\psi_1\colon N\to H^2/N\otimes K_X$ and $\theta$ in~$(1)$ and $\psi_2\colon H^2/N\to E/H^2\otimes K_X$ are induced by $\nabla\colon N\to H^2\otimes K_X$. In particular, if $\deg(N)=d_1$, then \mbox{$N=H^1$}, that is, the Simpson filtration coincides with the Hardar--Narasimhan filtration.
\end{itemize}
\end{itemize}
\end{itemize}
\end{itemize}
\end{Theorem}

\begin{proof}(1) and (2) are dual to each other, here we just prove~(1). Since $H^1\subset E$ is the maximal destabilizing subbundle, we have $\mu\big(H^1\big)=d_1>0$, and since $I$ is the sub line bundle of the semistable bundle $E/H^1$, we have $\mu(I)\leq\mu\big(E/H^1\big)=-\frac{d_1}{2}$. On the other hand, the induced map $\theta\colon H^1\to I\otimes K_X$ is non-zero, which gives $\deg(I)\geq d_1-2g+2$. These give the maximal bound of $d_1$ and $\deg(I)$ as follows:
\begin{gather*}
0<d_1\leq\frac{4g-4}{3},\qquad
d_1-2g+2\leq\deg(I)\leq-\frac{d_1}{2}.
\end{gather*}
Consider the induced graded Higgs bundle $\big(H^1\oplus E/H^1,\theta\big)$, it is graded stable if and only if $\mu(H^1\oplus I)<0$, that is, $\deg(I)<-d_1$. Therefore, if the conditions in~(1.1) are satisfied, the Simpson filtration is $0\subsetneq H^1\subsetneq E$, with the associated gr-stable Higgs bundle $\big(H^1\oplus E/H^1,\theta\big)$. When $\big(H^1\oplus E/H^1,\theta\big)$ is not semistable, then its maximal destabilizing subbundle is $H^1\oplus I$, which should satisfy $\mu(H^1\oplus I)>0$. By Simpson's iteration process, the next filtration is
\begin{gather*}
0\subsetneq H^1\subsetneq F^1\subsetneq E,
\end{gather*}
where \looseness=-1 $F^1=\operatorname{Ker}\big(E\to\frac{E/H^1}{I}\big)\subset E$ is a rank 2 subbundle. The associated graded Higgs bundle is
\[
\left(H^1\oplus I\oplus\frac{E/H^1}{I},\begin{pmatrix}
0&0&0\\ \varphi_1&0&0\\0&\varphi_2&0
\end{pmatrix}\right),
\]
with $\varphi_1\colon H^1\to I\otimes K_X$ is induced by $\theta$ and $\varphi_2\colon I\to\frac{E/H^1}{I}\otimes K_X$ is induced by $\nabla\colon F^1\to E\otimes K_X$. It is gr-stable if and only if $\mu\big(\frac{E/H^1}{I}\big)<0$, which can be divided into the two kinds of bounds for~$d_1$ and $\deg(I)$ as in~(1.2), in either case, the Simpson filtration is $0\subsetneq H^1\subsetneq F^1\subsetneq E$. But if the associated graded Higgs bundle is not gr-semistable, we should have $\mu\big(\frac{E/H^1}{I}\big)>0$, that is, $\deg(I)>-d_1$, in this case, $d_1$ and $\deg(I)$ have bounds $0<d_1<g-1, d_1-2g+2\leq\deg(I)<-d_1$. By Simpson's iteration process, the next filtration is $0\subsetneq H^1\subsetneq E$, which comes back to the case~(1.1), the associated graded Higgs bundle will be stable, and the iteration process stop here. Therefore, we finish the proof of~(1).

(3) Since $H^1\subset E$ is the maximal destabilizing subsheaf and $H^2/H^1\subset E/H^1$ is the maximal destabilizing subsheaf. And the induced morphisms $H^1\to I\otimes K_X$ and $H^2/N\to E/H^2\otimes K_X$ are both non-zero, we have the maximal bounds of $\deg(I)$ and $\deg(N)$ as follows:
\begin{gather*}
d_1-2g+2\leq\deg(I)\leq d_2,\qquad 2(d_1+d_2)-2g+2\leq\deg(N)\leq d_1.
\end{gather*}
By the uniqueness of maximal destabilizing subsheaf, $\deg(I)=d_2$ if and only if $I=H^2/H^1$, and $\deg(N)=d_1$ if and only if $N=H^1$.

Look at the graded Higgs bundle $\big(H^1\oplus E/H^1,\theta\big)$, where $\theta\colon H^1\to E/H^1\otimes K_X$ is the induced map from $\nabla\colon H^1\to E\otimes K_X$. It is stable if and only if $\mu\big(H^2/H^1\big)<0$ and $\mu\big(H^1\oplus I\big)<0$, that is, $d_2<0$ and $\deg(I)<-d_1$, this is the case~(3.2.1). When it is not stable, then the maximal destabilizing subsheaf should have positive slope, the possible maximal destabilizing subsheaf is $H^1\oplus I$ or $H^2/H^1$:
\begin{enumerate}\itemsep=-1pt
\item[(a)] If $\mu\big(H^1\oplus I\big)>\mu\big(H^2/H^1\big)$ and $\mu\big(H^1\oplus I\big)>0$, that is, $\max\{2d_2-d_1,-d_1\}<\deg(I)\leq d_2$, then $\big(H^1\oplus E/H^1,\theta\big)$ has maximal destabilizing subsheaf $H^1\oplus I$. By Simpson's iteration, the next filtration is $0\subsetneq H^1\subsetneq F^1\subsetneq E$ for $F^1=\operatorname{Ker}\big(E\to\frac{E/H^1}{I}\big)\subset E$ a rank~2 subbundle. Easy to see that its associated graded Higgs bundle is stable.
\item[(b)] If $\mu\big(H^2/H^1\big)>\mu\big(H^1\oplus I\big)$ and $\mu\big(H^2/H^1\big)>0$, that is, $d_2>0$ and $d_1-2g+2\leq\deg(I)<2d_2-d_1$, then $\big(H^1\oplus E/H^1,\theta\big)$ has maximal destabilizing subsheaf $H^2/H^1$. By Simpson's iteration, the next filtration is $0\subsetneq H^2\subsetneq E$, but we should discuss the stability of the graded Higgs bundle $\big(H^2\oplus E/H^2,\theta'\big)$:
\begin{itemize}\itemsep=-1pt
\item it is stable if and only if $\mu(N)<0$, that is, when $d_2>0$, $d_1-2g+2\leq\deg(I)<2d_2-d_1$ and $2(d_1+d_2)-2g+2\leq\deg(N)<0$, the Simpson filtration is $0\subsetneq H^2\subsetneq E$;
\item if it is not semistable, then its maximal destabilizing subsheaf is $N$ and should satisfy $\mu(N)>0$, and by Simpson's iteration, the next filtration is $0\subsetneq N\subsetneq H^2\subsetneq E$. Its associated Higgs bundle is stable, so the iteration stops.
\end{itemize}
\end{enumerate}

\vspace{-2mm}

Combining all the above, we obtain the statement~(3).
\end{proof}

For each $\alpha$, let $\big(G_\alpha^1\big)^{\rm VHS}\subset G_\alpha^1$ be the subset that consists of the polarized $\mathbb{C}$-VHS (complex variations of Hodge structure). It identifies to those Higgs bundles having the structure of systems of Hodge bundles (i.e., the fixed point set $P_\alpha\subset G_\alpha^0$) by the non-Abelian Hodge correspondence. Therefore, $\big(G_\alpha^1\big)^{\rm VHS}={\rm NAH}(P_\alpha)$. Simpson guessed in~\cite{CS8} that points in $G_\alpha^1\backslash\big(G_\alpha^1\big)^{\rm VHS}$ do not relate to the points in~$G_\alpha^0$ via the non-Abelian Hodge correspondence, this is the following conjecture:

\begin{Conjecture}
$\big(G_\alpha^1\big)^{\rm VHS}=G_\alpha^1\bigcap{\rm NAH}\big(G_\alpha^0\big)$.
\end{Conjecture}

For each polarized $\mathbb{C}$-VHS $(E,\nabla)$ such that the corresponding monodromy representation is irrreducible, then the Simpson filtration coincides with its Hodge filtration, this provides a~method to construct its Hodge filtration.

\vspace{-1.5mm}

\section[Twistor structures arising from non-Abelian Hodge correspondence]{Twistor structures arising from non-Abelian Hodge\\ correspondence}

Another important application of non-Abelian Hodge theory is the study of Hitchin's construction of twistor spaces for hyper-K\"ahler manifolds applied to the moduli spaces with hyper-K\"ahler structure (see Section~\ref{sec5.1}). This construction of twistor structures is interpreted by Deligne as a gluing of two moduli spaces, the resulting twistor space is shown being analytic isomorphic to the original one (see Section~\ref{sec5.2}). When we focus on compact Riemann surfaces, then Deligne's gluing can be generalized to a gluing of two moduli spaces induced from any element of the outer automorphism group of the fundamental group (see Section~\ref{sec5.3}).

In this section, we will introduce the twistor construction from Hitchin and Deligne's viewpoints. Then we will provide a new interpretation based on a recent work \cite{H}.

\subsection{Hitchin's twistor construction}\label{sec5.1}

In \cite{HKLR}, the authors provide a construction of twistor space for any hyper-K\"ahler manifold. Let~$M$ be a heper-K\"ahler manifold with three complex structures $(I,J, K=IJ)$. The stereographic projection $\mathbb{P}^1\to S^2$,
\begin{gather*}
\lambda=u+iv\mapsto\left(x=\frac{1-|\lambda|^2}{1+|\lambda|^2},\,y=\frac{2u}{1+|\lambda|^2},\, z=\frac{2v}{1+|\lambda|^2}\right)
\end{gather*}
defines a family of complex structures $I_\lambda:=xI+yJ+zK$ on~$M$.

The {\em twistor space} ${\rm TW}(M)$ of $M$ is a $C^\infty$ trivialization ${\rm TW}(M)\cong M\times\mathbb{P}^1$. $I_\lambda$ determines an almost complex structure $\mathcal{I}$ on ${\rm TW}(M)$, let $a=m\times\lambda\in{\rm TW}(M)$, then
\begin{gather*}
\mathcal{I}\colon \ T_a{\rm TW}(M)=T_mM\oplus T_\lambda\mathbb{P}^1\to T_mM\oplus T_\lambda\mathbb{P}^1
\end{gather*}
is given by $(I_\lambda,I_0)$, where $I_0\colon T_\lambda\mathbb{P}^1\to T_\lambda\mathbb{P}^1$ is the usual complex structure on $\mathbb{P}^1$ given by $I_0(v)=iv$. In fact, $\mathcal{I}$ is integrable~\cite{HKLR,Sal}, thus the twistor space~${\rm TW}(M)$ is a complex manifold of dimension $\dim_\mathbb{C}({\rm TW}(M))=1+\dim_\mathbb{C}(M)$.

The twistor space ${\rm TW}(M)$ has the following properties:
\begin{itemize}\itemsep=0pt
\item[(1)] the projection $\pi\colon {\rm TW}(M)\to\mathbb{P}^1$ is holomorphic;
\item[(2)] there is an antilinear involution $\sigma\colon {\rm TW}(M)\to{\rm TW}(M)$, $(m,\lambda)\mapsto\big(m,-\bar{\lambda}^{-1}\big)$, which covers the antipodal involution $\sigma_{\mathbb{P}^1}\colon \mathbb{P}^1\to\mathbb{P}^1$, $\lambda\mapsto-\bar{\lambda}^{-1}$, so it gives a real structure on~${\rm TW}(M)$;
\item[(3)] for any $m\in M$, the section $\{m\}\times\mathbb{P}^1\subset {\rm TW}(M)$ is holomorphic and $\sigma$-invariant, we call it a {\em preferred section}, in some papers like~\cite{HKLR,Sal}, it is called a {\em twistor line};
\item[(4)] weight 1 property: the normal bundle along any preferred section is isomorphic to \linebreak $\mathcal{O}_{\mathbb{P}^1}(1)^{\oplus\dim_\mathbb{C}(M)}$.
\end{itemize}

\begin{Proposition}\label{prop5.1}Preferred sections are $\sigma$-invariant. Moreover, locally, preferred sections are the only $\sigma$-invariant holomorphic sections.
\end{Proposition}

If we denoted by ${\rm Pre}$ the set of preferred sections in the Douaby moduli space ${\rm Sec}$ of holomorphic sections. Then ``locally" means there exists an open neighborhood $U$ of ${\rm Pre}$ in ${\rm Sec}$ such that preferred sections are the only $\sigma$-invariant sections in $U$ \cite{KV}. To my knowledge, whether Proposition~\ref{prop5.1} holds globally is unknown:

\begin{Question}Are the preferred sections the only $\sigma$-invariant sections?
\end{Question}

Moreover, let $p\colon {\rm TW}(M)\to M$ be the natural projection, then the twistor space ${\rm TW}(M)$ admits a natural hermitian metric $g:=p^*g_M+\pi^*g_{\rm FS}$, where $g_M$ is the hyper-K\"ahler metric on~$M$ and $g_{\rm FS}$ is the Fubini--Study metric on $\mathbb{P}^1$. This metric makes the twistor space ${\rm TW}(M)$ into a {\em balanced manifold}~\cite{KV} (i.e., the associated fundamental form $\omega(\bullet,\bullet):=g(\mathcal{I}\bullet,\bullet)$ satisfies the weak closedness condition $d(\omega^{\dim_\mathbb{C}M})=0$).

Let $X$ be a complex projective variety, denote by $M_{\rm dR}^{\rm sm}(X,r)$ and $M_{\rm Dol}^{\rm sm}(X,r)$ the open subsets of smooth points of the moduli spaces $M_{\rm dR}(X,r)$ and $M_{\rm Dol}(X,r)$, respectively. From non-Abelian Hodge theory, we have $C^\infty$ isomorphism $M_{\rm dR}^{\rm sm}(X,r)\cong M_{\rm Dol}^{\rm sm}(X,r)$ coming from the homeomorphism between the underlying topological spaces \cite{CS4,CS5,CS6}. Let $M^{\rm sm}(X,r)$ be the underlying differentiable manifold. Then by the work of Hitchin (\cite{NH} for 1-dimensional base mani\-fold) and Fujiki~(\cite{Fuj} for higher-dimensional K\"ahler mani\-folds as the base), $M^{\rm sm}(X,r)$~carries two complex structures~$I$ and~$J$ from the complex manifolds $\big(M_{\rm Dol}^{\rm sm}(X,r),I\big)$ and \linebreak $\big(M_{\rm dR}^{\rm sm}(X,r),J\big)$. Moreover, $IJ=:K$ gives the third complex structure, which makes\linebreak $\big(M^{\rm sm}(X,r), I, J, K\big)$ has the structure of a hyper-K\"ahler manifold. Hence Hitchin's idea for constructing the twistor space of hyper-K\"ahler manifolds can be applied to~$M^{\rm sm}(X,r)$. Therefore, we obtain a twistor space for~$M^{\rm sm}(X,r)$, denoted as ${\rm TW}\big(M^{\rm sm}(X,r)\big)$, and called the {\em Hitchin twistor space}.

\subsection{Deligne's interpretation}\label{sec5.2}

Deligne's original aim was to understand the Hitchin twistor space ${\rm TW}\big(M^{\rm sm}(X,r)\big)$ via $\lambda$-con\-nec\-tions. The idea is gluing the Hodge moduli space $M_{\rm Hod}(X,r)$ over $X$ and the Hodge moduli space $M_{\rm Hod}\big(\bar{X},r\big)$ over its conjugate chart $\bar{X}$ to obtain a twistor space ${\rm TW}_{\rm DH}(X,r)$ that is analytic isomorphic to ${\rm TW}\big(M^{\rm sm}(X,r)\big)$. This was described and further studied by Simpson, he interpreted $M_{\rm Hod}(X,r)$ as the Hodge filtration on the non-Abelian de Rham cohomology~$M_{\rm dR}(X,r)$. He showed the Griffiths transversality and the regularity of the Gauss--Manin connection for this filtration~\cite{CS6}. Here we introduce their ideas based on Simpson's papers~\cite{CS6,CS7}.

Let $X$ be a complex projective variety with a fixed base point~$x$. The complex conjugation gives a map $\varphi\colon X^{\rm top}\to\bar{X}^{\rm top}$ between the underlying topological spaces, which induces an isomorphism
\begin{gather}\label{5.1}
\varphi_*\colon \ \pi_1(X,x)\to\pi_1\big(\bar{X},\bar{x}\big).
\end{gather}
As the notations used in the last section, let $M_{\rm Hod}(X,r)$ be the moduli space of semistable $\lambda$-flat bundles of rank $r$ with vanishing Chern classes, let $M_{\rm B}(X,\mathrm{GL}(r,\mathbb{C}))$ be the moduli space of representations. The Riemann--Hilbert correspondence gives the analytic isomorphism~\cite{CS4}:
\begin{gather}\label{5.2}
M_{\rm dR}(X,r)\cong M_{\rm B}(X,r).
\end{gather}
The $\mathbb{C}^*$-action on $M_{\rm Hod}(X,r)$ gives the algebraic isomorphism:
\begin{align}
M_{\rm Hod}(X,r)\times_{\mathbb{C}}\mathbb{C}^*&\cong M_{\rm dR}(X,r)\times\mathbb{C}^*,\nonumber\\
\big[E,\bar{\partial}_E,D^\lambda,\lambda\big]&\leftrightarrow \big(\big[E,\bar{\partial}_E,\lambda^{-1}D^\lambda\big],\lambda\big).\label{5.3}
\end{align}
The gluing isomorphism $\mathbf{d}_\varphi\colon M_{\rm Hod}(X,r)\times_\mathbb{C}\mathbb{C}^*\to M_{\rm Hod}(\bar{X},r)\times_\mathbb{C}\mathbb{C}^*$ is given as follows:

First by isomorphisms \eqref{5.1}--\eqref{5.3}, any point $\big[E,\bar{\partial}_E,D^\lambda,\lambda\big]$ in $M_{\rm Hod}(X,r)\times_{\mathbb{C}}\mathbb{C}^*$ determines a~representation $\rho\big(\lambda^{-1}D^\lambda\big)\comp\varphi_*^{-1}\colon \pi_1\big(\bar{X},\bar{x}\big)\to {\rm GL}(r,\mathbb{C})$, where $\rho\big(\lambda^{-1}D^\lambda\big)\colon \pi_1(X,x)\to {\rm GL}(r,\mathbb{C})$ is the monodromy corresponds to the flat connection $\lambda^{-1}D^\lambda$. Then by the conjugate version of isomorphisms~\eqref{5.2}, $\rho\big(\lambda^{-1}D^\lambda\big)\comp\varphi_*^{-1}$ corresponds to a flat bundle over~$\bar{X}$. Finally, the conjugate version of~\eqref{5.3} at the fiber $\lambda^{-1}$ gives a $\lambda^{-1}$-flat bundle over~$\bar{X}$. The obtained $\lambda^{-1}$-flat bundle over $\bar{X}$ is $\big[E,\lambda^{-1}D^\lambda,\lambda^{-1}\bar{\partial}_E,\lambda^{-1}\big]$. Therefore, the gluing isomorphism is
\begin{align*}
\mathbf{d}_\varphi\colon \ M_{\rm Hod}(X,r)\times_\mathbb{C}\mathbb{C}^*&\cong M_{\rm Hod}\big(\bar{X},r\big)\times_\mathbb{C}\mathbb{C}^*,\\
\big[E,\bar{\partial}_E,D^\lambda,\lambda\big]& \leftrightarrow\big[E,\lambda^{-1}D^\lambda,\lambda^{-1}\bar{\partial}_E,\lambda^{-1}\big],
\end{align*}
which covers the map $\mathbb{C}^*\to\mathbb{C}^*, \lambda\mapsto\lambda^{-1}$. Therefore, we obtained a space ${\rm TW}_{\rm DH}(X,r)$, called the {\em Deligne--Hitchin twistor space}, together with a fibration ${\rm TW}_{\rm DH}(X,r)\to\mathbb{P}^1$ which extends the projection $\pi\colon M_{\rm Hod}(X,r)\to\mathbb{C}$. The fibers of this fibration are $M_{\rm Dol}(X,r)$ at $\lambda=0$, $M_{\rm Dol}\big(\bar{X},r\big)$ at $\lambda=\infty$, and analytic isomorphic to $M_{\rm dR}(X,r)$ at $\lambda\neq0,\infty$.

Let $\big[E,\bar{\partial}_E,\theta,h\big]\in M_{\rm Dol}(X,r)$ be a harmonic Higgs bundle with the pluri-harmonic metric~$h$, then it determines a holomorphic section $p\colon \mathbb{P}^1\to{\rm TW}_{\rm DH}(X,r)$:
\begin{gather*}
\lambda\longmapsto\big[E,\bar{\partial}_E+\lambda\theta_h^\dagger,\lambda\partial_{E,h}+\theta,\lambda\big],
\end{gather*}
where $\theta_h^\dagger$ and $\partial_{E,h}$ are the unique operators determined by~$(\theta,h)$ and $\big(\bar{\partial}_E,h\big)$, respectively. This section is called a {\em preferred section}.

We can also define an antilinear involution $\sigma\colon {\rm TW}_{\rm DH}(X,r)\to{\rm TW}_{\rm DH}(X,r)$ that covers the antipodal involution $\sigma_{\mathbb{P}^1}\colon \mathbb{P}^1\to\mathbb{P}^1$. This map is defined by gluing the antiholomorphic ismorphisms $\sigma_{\rm Hod,X}\colon M_{\rm Hod}(X,r)\to M_{\rm Hod}\big(\bar{X},r\big)$ and $\sigma_{\rm Hod, \bar X}\colon M_{\rm Hod}\big(\bar{X},r\big)\to M_{\rm Hod}(X,r)$, where
\begin{align*}
\sigma_{\rm Hod,X}\colon \ M_{\rm Hod}(X,r)&\to M_{\rm Hod}\big(\bar{X},r\big),\\
\big[E,\bar{\partial}_E,D^\lambda,\lambda\big]& \mapsto\big[\bar{E}^*,\overline{\bar{\partial}_E}^*,\overline{D^\lambda}^*,-\bar{\lambda}\big].
\end{align*}
By the existence of pluri-harmonic metric for (poly-)stable $\lambda$-flat bundles with vanishing Chern classes (see Theorem~\ref{thm2.6}), take $h$ being the pluri-harmonic metric associated to $\big[E,\bar{\partial}_E,D^\lambda,\lambda\big]$. Then we have the isomorphism
\begin{gather*}
\big[\bar{E}^*,\overline{\bar{\partial}_E}^*, \overline{D^\lambda}^*,-\bar{\lambda}\big]\cong\big[E,\delta_h',-\delta_h'',-\bar{\lambda}\big],
\end{gather*}
where the operators $\delta_h'$ and $\delta_h''$ can be found in the second section.

Therefore, by means of pluri-harmonic metric, we can write the involution $\sigma$ as follows:
\begin{align*}
\sigma\colon \ {\rm TW}_{\rm DH}(X,r)&\to{\rm TW}_{\rm DH}(X,r),\\
\big[E,\bar{\partial}_E,D^\lambda,\lambda\big]&\mapsto \big[E,\bar{\lambda}^{-1}\delta_h'',-\bar{\lambda}^{-1}\delta_h',-\bar{\lambda}^{-1}\big].
\end{align*}
In fact, $\sigma$ is the product of the following 3 involutions~\cite{CS7}:
\begin{itemize}\itemsep=0pt
\item[(1)] an antiholomorphic involution
\begin{align*}
C\colon \ {\rm TW}_{\rm DH}(X,r)&\to{\rm TW}_{\rm DH}(X,r),\\
\big[E,\bar{\partial}_E,D^\lambda,\lambda\big]&\mapsto \big[E,\bar{\lambda}^{-1}\overline{D^\lambda}, \bar{\lambda}^{-1}\overline{\bar{\partial}_E},\bar{\lambda}^{-1}\big]
\end{align*}
obtained by gluing complex conjugations of $\lambda$-flat bundles;
\item[(2)] a holomorphic involution
\begin{align*}
D\colon \ {\rm TW}_{\rm DH}(X,r)&\to{\rm TW}_{\rm DH}(X,r),\\
\big[E,\bar{\partial}_E,D^\lambda,\lambda\big]& \mapsto\big[E^*,\big(\bar{\partial}_E\big)^*,\big(D^\lambda\big)^*,\lambda\big]
\end{align*}
obtained by gluing duals of $\lambda$-flat bundles;
\item[(3)] a holomorphic involution
\begin{align*}
N\colon \ {\rm TW}_{\rm DH}(X,r)&\to{\rm TW}_{\rm DH}(X,r),\\
\big[E,\bar{\partial}_E,D^\lambda,\lambda\big]&\mapsto\big[E,\bar{\partial}_E,-D^\lambda,-\lambda\big]
\end{align*}
obtained by $-1\in\mathbb{C}^*$ acts on ${\rm TW}_{\rm DH}(X,r)$.
\end{itemize}

By definition, any point in the twistor space determines a unique preferred section. Therefore, the set of preferred sections gives a homeomorphism
\begin{gather*}
{\rm TW}_{\rm DH}(X,r)\cong M_{\rm Dol}(X,r)\times\mathbb{P}^1,
\end{gather*}
which is a $C^\infty$ isomorphism over smooth points \cite{CS6}:
\begin{gather*}
{\rm TW}_{\rm DH}^{\rm sm}(X,r)\cong M_{\rm Dol}^{\rm sm}(X,r)\times\mathbb{P}^1.
\end{gather*}

\begin{Proposition}\label{prop5.3}Preferred sections are $\sigma$-invariant. Moreover, locally, preferred sections are the only $\sigma$-invariant holomorphic sections.
\end{Proposition}

Preferred sections are invariant under the antilinear involution $\sigma$ is easy to see. Here ``locally'' means that for any given preferred section $p\colon \mathbb{P}^1\to{\rm TW}_{\rm DH}(X,r)$, there exists an open neighbourhood $U\subseteq{\rm TW}_{\rm DH}(X,r)$ of $p$ such that preferred sections are the only $\sigma$-invariant sections in~$U$. This is true because the normal bundle along any preferred section is isomorphic to $\big(\mathcal{O}_{\mathbb{P}^1}(1)\big)^{\oplus\dim_\mathbb{C}M_{\rm dR}(X,r)}$.

In \cite{CS6}, Simpson asked the following question:

\begin{Question}Does Proposition~{\rm \ref{prop5.3}} hold globally? That is, are the preferred sections the only $\sigma$-invariant sections?
\end{Question}

In \cite{CS6,CS7}, Simpson showed this question is true for twistor space of rank~1 bundles. Recently, in~\cite{BHR}, the authors constructed holomorphic $\sigma$-invariant but not preferred sections for twistor space of rank~2 bundles over compact Riemann surface of~$g\geq2$.

Now we give a brief conclusion on above two different kinds of constructing the twistor spaces. In Hitchin's constructing, we first have a hyper-K\"ahler manifold (the moduli spaces are hyper-K\"ahler) $M^{\rm sm}(X,r)$, then the twistor space has a naturally structure of $M^{\rm sm}(X,r)\times\mathbb{P}^1$. The complex structure is induced from a family of complex structures (given by the quaternionic structure) on $M^{\rm sm}(X,r)$ and the natural complex structure on $\mathbb{P}^1$. While from Deligne's interpretation, the twistor space is obtained by gluing the moduli spaces $M_{\rm Hod}(X,r)$ over $X$ and $M_{\rm Hod}\big(\bar{X},r\big)$ over the conjugate $\bar{X}$. This gluing is obtained via the algebraic isomorphism between $M_{\rm Hod}^{\lambda}$ ($\lambda\neq0$) and $M_{\rm dR}$, the isomorphism of fundamental groups induced from the complex conjugate map $X\to\bar{X}$, and the analytic isomorphism $M_{\rm dR}\cong M_{\rm B}$ given by the Riemann--Hilbert correspondence. Preferred sections give the twistor space the product structure, and the weight 1 property implies the quaternionic structure. This viewpoint shows that a quaternionic structure on $M^{\rm sm}(X,r) $ is equivalent to the weight~1 property of a preferred section~\cite{CS7}.

Moreover, the twistor spaces arising from these two different methods are analytic isomorphic:

\begin{Theorem}[\cite{CS6}]The twistor space ${\rm TW}^{\rm sm}_{\rm DH}(X,r)$ is analytic isomorphic to ${\rm TW}(M^{\rm sm}(X,r))$.
\end{Theorem}

\subsection{A new interpretation}\label{sec5.3}

In \cite{H}, we concentrated on the twistor space of Riemann surface case, we generalized Deligne's construction for any element~$\gamma$ of the outer automorphism group of the fundamental group of the Riemann surface to obtain the $\gamma$-twistor space, we give a brief introduction here (for more details and proofs, see~\cite{H}).

Let $X$ be a compact Riemann surface of genus $g\geq2$ and let~$\mathcal{X}$ be the underlying smooth surface, so $X$ can be wrote as $X=(\mathcal{X},I)$ for a complex structure $I$ that defines the Riemann surface structure of $X$. The fundamental group of~$\mathcal{X}$ is given by
\begin{gather*}
\pi_1(\mathcal{X})=\bigg\langle\alpha_1,\beta_1,\dots,\alpha_g,\beta_g\colon \prod_{i=1}^g\alpha_i\beta_i\alpha_i^{-1}\beta_i^{-1}=1\bigg\rangle.
\end{gather*}

Take $G=\mathrm{GL}(r,\mathbb{C})$ and let $M_{\mathrm{B}}(\mathcal{X},G)$ be the coarse moduli space of rank $r$ representations of~$\pi_1(\mathcal{X})$ into~$G$.

When $\lambda\neq 0$, we have the analytic isomorphism
\begin{gather*}
\mu_\lambda\colon \ M^\lambda_{{\rm Hod}}(X)\xlongrightarrow{\cong} M_\mathrm{B}(\mathcal{X},G)
\end{gather*}
given by the composition of the algebraic isomorphism $M^\lambda_{{\rm Hod}}(X)\cong M_{\mathrm{dR}}(X,r)$ and the Riemann--Hilbert correspondence $M_{\mathrm{dR}}(X,r)\cong M_{\mathrm{B}}(\mathcal{X},G)$, where the first algebraic isomorphism is given
by rescaling the twistor parameter $\lambda$ (does not change the holomorphic structures of the underlying bundles).

The outer automorphism group
\begin{gather*}
\Gamma_\mathcal{X}:=\mathrm{Out}(\pi_1(\mathcal{X})) =\operatorname{Aut}(\pi_1(\mathcal{X}))/\mathrm{Inn}(\pi_1(\mathcal{X}))
\end{gather*}
 acts on $M_\mathrm{B}(\mathcal{X},G)$. It is known that $\Gamma_\mathcal{X}$ is isomorphic to the extended mapping class group
 \begin{gather*}
 {\rm Mod}^\diamondsuit(\mathcal{X}):=\pi_0(\mathrm{Diff}(\mathcal{X})) =\mathrm{Diff}(\mathcal{X})/\mathrm{Diff}_0(\mathcal{X})
 \end{gather*}
 that acts on the Teichm\"{u}ller space ${\rm Teich}(\mathcal{X})$ of $\mathcal{X}$. To define the Deligne gluing, we first define an action of $\Gamma_\mathcal{X}$ on ${\rm Teich}(\mathcal{X})\times M_\mathrm{B}(\mathcal{X},G) $ as follows: for $f\in \mathrm{Diff}(\mathcal{X})$
 such that the equivalence class $[f]\in{\rm Mod}^\diamondsuit(\mathcal{X})$ is nontrivial, there is a induced isomorphism
 \begin{gather*}
 f_*\colon \ \pi_1(\mathcal{X},x)\rightarrow \pi_1(\mathcal{X},f(x))
 \end{gather*}
 for $x\in \mathcal{X}$
 such that $[f_*]\in \Gamma_\mathcal{X}$ is nontrivial, the action of~$f$ on ${\rm Teich}(\mathcal{X})$ maps $X=(\mathcal{X},I)$ to $X^\prime=(\mathcal{X},I^\prime)$ where $I$ is a complex structure on $\mathcal{X}$ and $I^\prime $ is another complex structure induced by $f$, and the action of $f_*$ on $M_\mathrm{B}(\mathcal{X},G)$ maps a representation $\rho\colon \pi_1(\mathcal{X},x)\to G$ to another representation $\rho'=\rho\circ f_*^{-1}\colon \pi_1(\mathcal{X},f(x))\to G$, thus the action of $\Gamma_\mathcal{X}$ on ${\rm Teich}(\mathcal{X})\times M_\mathrm{B}(\mathcal{X},G) $ sends $(X,\rho)$ to $(X',\rho')$.

Choose $\gamma\in \Gamma_\mathcal{X}$. For $\big[E,\bar{\partial}_E,D^\lambda,\lambda\big]\in M_{{\rm Hod}}(X,r)\times_{\mathbb{C}}\mathbb{C}^*$, we have an $(X,\rho)\in {\rm Teich}(\mathcal{X})\times M_\mathrm{B}(\mathcal{X},G)$ via the isomorphism $\mu_\lambda$, then the $\gamma$-action maps $(X,\rho)$ to $(X',\rho')$ which corresponds to some $\big[E',\bar{\partial}_{E'},D^{'\lambda^{-1}},\lambda^{-1}\big]\in M_{{\rm Hod}}(X', r)\times_{\mathbb{C}}\mathbb{C}^*$ via $\mu_\lambda^{-1}$ “at the point $\lambda^{-1}\in\mathbb{C}^*$.

Therefore, we define an analytic isomorphism called the Deligne isomorphsim
\begin{gather*}
\mathbf{d}_\gamma\colon \ M_{{\rm Hod}}(X,r)\times_{\mathbb{C}}\mathbb{C}^*\rightarrow M_{{\rm Hod}}\big(X^\prime,r\big)\times_{\mathbb{C}}\mathbb{C}^*
\end{gather*}
that covers the map $\mathbb{C}^*\to\mathbb{C}^*$, $\lambda\mapsto \lambda^{-1}$. Now we can use this isomorphism $d_\gamma$ to glue together two analytic spaces $M_{{\rm Hod}}(X,r)$ and $M_{{\rm Hod}}(X',r)$ along their open sets. The resulting space is denoted by $\mathrm{TW}^\gamma(X,r)$, called the {\em $\gamma$-twistor space}, it has a fibration ${\rm TW}^\gamma(X,r)\to\mathbb{P}^1$. Obviously, this construction is independent of the choice of representative of $\gamma$ up to isomorphism. In conclusion, our construction is along the following diagram:
\[\xymatrix{
\mathcal{M}_{\rm Hod}(X,r)|_{\lambda}\ (\lambda\in\mathbb{C}^*) \ar[d]^{ \cong_{\rm alg}} \ar@{-->}[rr]^{\cong_{\rm an} \ \mathbf{d}_\gamma} &&\mathcal{M}_{\rm Hod}(X',r)|_{\lambda^{-1}}\ (\lambda\in\mathbb{C}^*) \\
\mathcal{M}_{\rm dR}(X,r)\ar[d]^{\cong_{\rm an}} &&\mathcal{M}_{\rm dR}(X',r)\ar[u]_{\cong_{\rm alg}}\\
\mathcal{M}_{\rm B}(\mathcal{X},r)\ar@{=}^{\bullet\comp f_*^{-1}}[rr]&&\mathcal{M}_{\rm B}(\mathcal{X},r).\ar[u]_{\cong_{\rm an}}}
\]

Taking $X=(\mathcal{X},I)$ and $X'=(\mathcal{X},-I)$, let $\varphi\colon X\to X', x\mapsto\bar{x}$ be the complex conjugate map, and $\varphi'\colon \mathcal{X}\to\mathcal{X}$ be the associated map on underlying smooth surface. Then $\varphi'$ is not an identity map and lies in a non-trivial class of ${\rm Mod}^\diamondsuit(\mathcal{X})$. This induces an isomorphism $\varphi_*\colon \pi_1(\mathcal{X},x)\to\pi_1(\mathcal{X},\varphi'(x))$, which lies in a non-trivial class of the outer automorphism group $\mathrm{Out}(\pi_1(\mathcal{X}))$. Therefore, the Deligne--Hitchin twistor space $\mathrm{TW}_{\mathrm{DH}}(X,r)$ can be thought as a~special $\gamma$-twistor space.

Fix a point $\big[E,\bar{\partial}_E,D^{\lambda_0},\lambda_0\big]\in M_{\rm Hod}(X,r)\times_{\mathbb{C}}\mathbb{C}^*$ for some fixed $\lambda_0\in\mathbb{C}^*$, then $u$ determines a~holomorphic section $s_{\lambda_0}\colon \mathbb{C}^*\to M_{\rm Hod}(X,r)\times_{\mathbb{C}}\mathbb{C}^*$ as follows:
\begin{gather*}
\lambda\longmapsto\big[E,\bar{\partial}_E,\lambda\lambda_0^{-1}D^{\lambda_0},\lambda\big].
\end{gather*}
This section can be extended to a~holomorphic section $\mathbb{P}^1\to{\rm TW}^\gamma(X,r)$ of the $\gamma$-twistor space by Simpson's Theorem~\ref{thm4.3} (see also~\cite{H}), the extended section is still denoted as $s_{\lambda_0}$, called the {\em de Rham section} of~${\rm TW}^\gamma(X,r)$.

The following property is obtained in~\cite{H}, where the Torelli-type theorem for the $\gamma$-twistor space is obtained by applying the techniques in~\cite{BGHL}, where the authors obtained the property for the Deligne--Hitchin twistor space.

\begin{Theorem}\quad
\begin{enumerate}\itemsep=0pt
\item[$(1)$] The de Rham section $s_{\lambda_0}$ also has the weight~$1$ property, that is, we have the following isomorphism:
\begin{align*}
N_{s_{\lambda_0}}\simeq\mathcal{O}_{\mathbb{P}^1}(1)^{\oplus\dim_{\mathbb{C}}(M_{\rm dR}(X, r))},
\end{align*}
where $N_{s_{\lambda_0}}$ denotes the normal bundle.
In particular, the $\gamma$-twistor space $\mathrm{TW}^\gamma(X,r)$ contains ample rational curves.
\item[$(2)$] $($Torelli-type theorem$)$ Let $X$, $Y$ be two compact Riemann surfaces with genus $g\geq3$. If we have the analytical isomorphism of $\gamma$-twistor spaces $\mathrm{TW}^\gamma(X,r)\cong\mathrm{TW}^\gamma(Y,r)$, then as Riemann surfaces, either $X\cong Y$, or $X\cong Y'$, where $Y'$ is the Riemann surface determined by $Y$ and $\gamma$.
\end{enumerate}
\end{Theorem}

\subsection*{Acknowledgements}

The author would like to thank his thesis supervisor Professor Carlos Simpson for much help on the understanding of this subject and for the kind guidance and useful discussions, and thank Professor Jiayu Li for continuous encouragement. Moreover, the author would like to express his deep appreciation to the anonymous referees for many valuable suggestions. The author is partial supported by China Scholarship Council (No.~201706340032).

\pdfbookmark[1]{References}{ref}
\LastPageEnding

\end{document}